\documentclass[11pt,draft]{article}
\usepackage{amsfonts,amsmath,amssymb,amsthm,cite}
\usepackage[mathscr]{eucal}
\usepackage[active]{srcltx}
\usepackage{calrsfs}
\usepackage{color}
\theoremstyle{plain}
\newtheorem{thm}{Theorem}[section]

\theoremstyle{definition}
\newtheorem{rem}[thm]{Remark}

\numberwithin{equation}{section}

\newtheorem{theorem}{Theorem}[section]

\numberwithin{equation}{section}

\newcommand{\beqn}{\begin{eqnarray}}
\newcommand{\eeqn}{\end{eqnarray}}
\newcommand{\be}{\begin{equation}}
\newcommand{\ee}{\end{equation}}
\newcommand{\ben}{\begin{enumerate}}
\newcommand{\een}{\end{enumerate}}

\def\Na{\mathbb N}
\def\Ze{\mathbb Z}

\def\esssup{\mathop{\text{\rm ess\,sup}}}

\def\Msp{\mathfrak{M}^+}

\begin{document}

\title{Weighted inequalities for quasilinear integral operators on the semiaxis and application to the Lorentz spaces}
\author{Dmitrii V.~Prokhorov,
Vladimir D. Stepanov\footnote{ The research work of D.V. Prokhorov and V.D. Stepanov was financially supported by the Russian Scientific Fund (Project
14-11-00443).}}

\date{}
\maketitle

\noindent{\bf Abstract:} Weighted $L^p-L^r$ inequalities with arbitrary measurable non-negative weights for positive quasilinear integral operators with Oinarov's kernel on the semiaxis are characterized. Application to the boundedness of maximal operator in the Lorentz $\Gamma-$spaces is given.  
\vspace{3mm}

\vspace{3mm}

\vspace{3mm}

\section{Introduction\label{Int}}

Let $\mathbb{R}_+:=[0,\infty).$ Denote $\mathfrak{M}$ the set of all  measurable functions on $\mathbb{R}_+$ and $\mathfrak{M}^+\subset\mathfrak{M}$ the subset of all non-negative functions. If $0<p\leq\infty$ and $v\in\mathfrak{M}^+$ we define
$$
L^p_v:=\left\{f\in\mathfrak{M}:\|f\|_{L^p_v}
:=\left(\int_0^\infty |f(x)|^pv(x)dx\right)^{\frac{1}{p}}<\infty\right\},
$$ 
\begin{equation*}
L^\infty_v:=\left\{f\in\mathfrak{M}:\|f\|_{L^\infty_v}
:=\underset{x\geq 0}{\mathrm{ess\,sup}}~v(x)|f(x)|<\infty\right\}.
\end{equation*}

Let $0<q\leq\infty$ and $w\in\mathfrak{M}^+.$
We consider quasilinear operators on $\mathfrak{M}^+$ of the form
\begin{equation*}
 (Tf)(x) =\left(\int_x^\infty w(y)\left(\int_0^y k(y,z)f(z)dz\right)^qdy\right)^{\frac{1}{q}},
\end{equation*}
\begin{equation*}
 ({\cal T}f)(x) =\left(\int_0^x w(y)\left(\int_y^\infty k(z,y)f(z)dz\right)^qdy\right)^{\frac{1}{q}},
\end{equation*}
\begin{equation*}
 (Sf)(x) =\left(\int_x^\infty w(y)\left(\int_y^\infty k(z,y)f(z)dz\right)^qdy\right)^{\frac{1}{q}},
\end{equation*}
\begin{equation*}
 ({\cal S}f)(x) =\left(\int_0^x w(y)\left(\int_0^y k(y,z)f(z)dz\right)^qdy\right)^{\frac{1}{q}}
\end{equation*}
and
\begin{equation*}
 ({\bf T}f)(x) =\left(\int_x^\infty k(y,x)w(y)\left(\int_0^y f(z)dz\right)^qdy\right)^{\frac{1}{q}},
\end{equation*}
\begin{equation*}
 ({\bf{\mathfrak T}}f)(x) =\left(\int_0^x k(x,y)w(y)\left(\int_y^\infty f(z)dz\right)^qdy\right)^{\frac{1}{q}},
\end{equation*}
\begin{equation*}
 ({\bf S}f)(x) =\left(\int_x^\infty k(y,x)w(y)\left(\int_y^\infty f(z)dz\right)^qdy\right)^{\frac{1}{q}},
\end{equation*}
\begin{equation*}
 ({\mathfrak S}f)(x) =\left(\int_0^x k(x,y)w(y)\left(\int_0^y f(z)dz\right)^qdy\right)^{\frac{1}{q}},
\end{equation*}
where $k(x,y)\geq 0$ is a measurable kernel and the right hand sides are to replace by essential supremums
\begin{equation*}
 (Tf)(x) =\underset{y\geq x}{\mathrm{ess\,sup}}~ w(y)\int_0^y k(y,z)f(z)dz,
\end{equation*}
\begin{equation*}
({\bf T}f)(x) =\underset{y\geq x}{\mathrm{ess\,sup}}~ k(y,x)w(y)\int_0^y f(z)dz,
\end{equation*}
and similarly for the others, when $q=\infty.$ 

Let $u, v, w\in \mathfrak{M}^+$ be weights, $1\leq p\leq\infty,$ $0<r\leq\infty.$  Our aim is to characterize the weighted inequalities
\begin{equation}\label{Tin}
 \left\|Tf\right\|_{L^r_u} \leq C_T\left\|f\right\|_{L^p_v},\,\,f\in\mathfrak{M}^+,
\end{equation}
\begin{equation}\label{calTin}
 \left\|{\mathcal T}f\right\|_{L^r_u} \leq C_{{\mathcal T}}\left\|f\right\|_{L^p_v},\,\,f\in\mathfrak{M}^+,
\end{equation}
\begin{equation}\label{Sin}
 \left\|Sf\right\|_{L^r_u} \leq C_S\left\|f\right\|_{L^p_v},\,\,f\in\mathfrak{M}^+,
\end{equation}
\begin{equation}\label{calSin}
 \left\|{\mathcal S}f\right\|_{L^r_u} \leq C_{{\mathcal S}}\left\|f\right\|_{L^p_v},\,\,f\in\mathfrak{M}^+
\end{equation}
and
\begin{equation}\label{bTin}
 \left\|{\bf T}f\right\|_{L^r_u} \leq C_{{\bf T}}\left\|f\right\|_{L^p_v},\,\,f\in\mathfrak{M}^+,
\end{equation}
\begin{equation}\label{fTin}
 \left\|{\mathfrak T}f\right\|_{L^r_u} \leq C_{{\mathfrak T}}\left\|f\right\|_{L^p_v},\,\,f\in\mathfrak{M}^+,
\end{equation}
\begin{equation}\label{bSin}
\left\|{\bf S}f\right\|_{L^r_u} \leq C_{{\bf S}}\left\|f\right\|_{L^p_v},\,\,f\in\mathfrak{M}^+,
\end{equation}
\begin{equation}\label{fSin}
 \left\|{\mathfrak S}f\right\|_{L^r_u} \leq C_{\mathfrak S}\left\|f\right\|_{L^p_v},\,\,f\in\mathfrak{M}^+,
\end{equation}
where  a Borel function $k(x,y)\geq 0$ on $[0,\infty)^2$ satisfies Oinarov's condition: $k(x,y)=0$ if  $x<y$, and there is a constant $D\geq 1$ independent of $x\geq z\geq y\geq 0$ such that
\begin{equation}\label{O}
\frac{1}{D}\left(k(x,z)+k(z,y)\right)\leq k(x,y)\leq D\left(k(x,z)+k(z,y)\right)
\end{equation} 
and the constants $C_T$ and others are taken as the least possible. If $q=r<\infty$ these inequalities are reduced to the generalized Hardy-type inequalities which were well studied 
see, for instance, \cite{BK},  \cite{O}, \cite{S1} with further extensions and improvements in  \cite{Lai}, \cite{LS},  \cite{P0}, \cite{P1},  \cite{P4}, \cite{SU1}, \cite{SU2} and others. The case $q=\infty$ is closely related to recently initiated  studies of supremum operators \cite{GOP}, \cite{GP}, \cite{P2}, \cite{P3}, \cite{P5}, \cite{PS2}, \cite{S2}. If $k(x,y)\equiv 1$ the inequality \eqref{calSin} plays an important role in analysis on the Morry-type spaces (see, \cite{BG}, \cite{BGGM1}, \cite{BGGM2}, \cite{BJT}, \cite{BO}. In particular, for some parameters $p, q, r$
this case of \eqref{calTin} was solved in \cite{GMP1}, \cite{GMP2} and \eqref{calSin} in \cite{BO}. Complete solution of this case is given in \cite{PS3}, \cite{PS4}.

By a new method we characterize  the inequalities \eqref{Tin}--\eqref{fSin} with a kernel $k(x,y)$ satisfying \eqref{O} for all parameters $1\leq p\leq\infty, 0<r\leq\infty, 0<q\leq\infty.$ The cases $p=\infty$ and $r=\infty$ are trivial and the interval $0<p<1$ is excluded because in this case it can be shown that if, say, $C_T<\infty,$ then $C_T=0$ (see \cite{PS1}, Theorem 2 for details).  

Sections \ref{TS} and \ref{calTS} are devoted to the study of \eqref{Tin}--\eqref{calSin} and sections \ref{bTS} and \ref{mfTS} to \eqref{bTin}--\eqref{fSin}. It is interesting to observe that the second part is partially based on the first. In the last section \ref{M} we illustrate our results by a solution of well known problem on a sharp characterization of the $\Gamma^p(v)\to\Gamma^q(w)$ boundedness of the Hardy-Littlewood maximal operator for all $0<p,q<\infty$ including the most difficult cases missed in \cite{GHS} and \cite{S}.
 
We use signs $:=$ and $=:$ for determining new quantities and $\mathbb{Z}$ for the set of all integers. For positive functionals $F$ and $G$ we write $F\lesssim G,$ if $F\leq cG$ with some positive constant $c$, which depends only on irrelevant parameters. $F\approx G$ means $F\lesssim G\lesssim F$ or $F=cG.$ $\chi_E$ denotes the characteristic function (indicator) of a set $E.$ Uncertainties of the form $0\cdot\infty, \frac{\infty}{\infty}$ and $\frac{0}{0}$ are taken to be zero. 
 $\Box$ stands for the end of a proof. 

\section{Operators $T$ and $S$\label{TS}}

Suppose for simplicity that $\int_0^t u<\infty$ for all $t>0$ and define
the functions $\sigma:[0,\infty)\to [0,\infty]$, $\sigma^{-1}:[0,\infty)\to [0,\infty)$ by (here $\inf\varnothing=\infty$) 
$$
\sigma(x):=\inf\left\{y>0:\int_0^y u\ge 2\int_0^x u\right\},~~
\sigma^{-1}(x):=\inf\left\{y>0:\int_0^y u\ge \frac{1}{2}\int_0^x u\right\}.$$
The functions $\sigma$ and $\sigma^{-1}$ are increasing and from the continuity of an integral with respect to an upper limit it follows for any $x\in[0,\infty)$ that
$\int_0^{\sigma^{-1}(x)}u=\frac{1}{2}\int_0^x u$ and if $\sigma(x)<\infty$, then  $\int_0^{\sigma(x)}u=2\int_0^x u$.

Let $\sigma^{m}$, $m\in \Na$ be a composition of $m$ functions $\sigma$ and similar for  $\sigma^{-m}.$

For $0<c<d\le\infty$ and $f\in\Msp$ we put
\begin{align*}(H_{c,d} f)(x)&:=\chi_{[c,d)}(x)\int_{\sigma^{-1}(c)}^x k(x,z)f(z)dz,\\
 (H_c f)(x)&:=\chi_{[c,\infty)}(x)\int_0^x k(x,z)f(z)dz.
\end{align*}

 
\begin{thm}\label{theorem11} Let $1\le p<\infty$,  $0<r<\infty$, $0<q\le\infty$, $\frac{1}{s}:=\left(\frac{1}{r}-\frac{1}{p}\right)_+$. For validity of the inequality
 \eqref{Tin} it is necessary and sufficient that the inequalities 
\begin{align}
 &\left(\int_0^\infty u(x)\left(\int_x^\infty w\right)^\frac{r}{q}\left(\int_0^x k(x,z)f(z)dz\right)^r\,dx\right)^\frac{1}{r}\le A_0\|f\|_{L^p_v},\label{A0} \\
&\left(\int_0^\infty u(x)\left(\int_x^\infty [k(z,x)]^qw(z)dz\right)^\frac{r}{q}\left(\int_0^x f\right)^r\,dx\right)^\frac{1}{r}\le A_1\|f\|_{L^p_v},
\label{A1} 
   \end{align}
if $q<\infty$ or
\begin{align}
&\left(\int_0^\infty u(x)[\underset{y\geq x}{\mathrm{ess\,sup}}~w(y)]^r\left(\int_0^x k(x,z)f(z)dz\right)^r\,dx\right)^\frac{1}{r}\le A_0\|f\|_{L^p_v},\label{A0inf} \\
&\left(\int_0^\infty u(x)[\underset{y\geq x}{\mathrm{ess\,sup}}[w(y)k(y,x)]]^r\left(\int_0^x f\right)^r\,dx\right)^\frac{1}{r}\le A_1\|f\|_{L^p_v}
\label{A1inf} 
   \end{align}   
for $q=\infty$ hold for all $f\in \Msp$ and the constant
 \begin{equation}\label{A2}
A_2:=\begin{cases}
    \displaystyle\sup_{t\in(0,\infty)}\left(\int_0^t u\right)^\frac{1}{r}\|H_t\|_{L^p_v\to L^q_w}, & p\le r,\\
    \displaystyle \left(\int_0^\infty u(x)\left(\int_0^x u\right)^\frac{s}{p}\|H_{\sigma^{-1}(x),\sigma(x)}\|^s_{L^p_v\to L^q_w}\,dx\right)^\frac{1}{s}, & r<p
   \end{cases}
 \end{equation}
is finite. Moreover, $C_T\approx A_0+A_1+A_2.$
\end{thm}
 
\begin{proof} Let $n_0\in\Ze$ be such an integer that $2^{n_0}<\int_0^\infty u$. Put
\begin{align*}
a_{n_0}&:=\inf\left\{y>0:\int_0^y u\ge 2^{n_0}\right\},\\
a_{n+1}&:=\sigma(a_n)~\text{for}~n\ge n_0,\\
a_{n-1}&:=\sigma^{-1}(a_n)~\text{for}~n\le n_0.
\end{align*}
Denote $N:=\sup\{n\in\Ze:a_n<\infty\}.$ If  $N<\infty$ we put $a_{N+1}:=\infty.$ Observe, that $a_{n-1}=\sigma^{-1}(a_n)$ and $\sigma(a_n)=a_{n+1}$ for all $n\leq N.$ 

We suppose first that $q<\infty.$

{\it Sufficiency.} Let $\Delta_n:=[a_n,a_{n+1}).$ Applying the condition \eqref{O} and the relation (\cite{GHS}, Proposition 2.1)
\begin{equation}\label{sum}
\sum_{n\in\mathbb{Z}} 2^n\left(\sum_{i\geq n}\lambda_i\right)^s\approx \sum_{n\in\mathbb{Z}} 2^n\lambda_n^s,
\end{equation}
which is valid for all sequences $\{\lambda_n\}$ of non-negative numbers and any $s>0,$
we have
\begin{align*} 
\int_0^\infty [Tf]^ru&=\sum_{n\leq N}\int_{\Delta_n}[Tf]^ru
\approx \sum_{n\leq N}2^n\left(\int_{a_n}^\infty w(y)\left(\int_0^y k(y,z)f(z)dz\right)^qdy\right)^{\frac{r}{q}}\\
&\approx \sum_{n\leq N}2^n\left(\int_{\Delta_n}w(y)\left(\int_0^y k(y,z)f(z)dz\right)^qdy\right)^{\frac{r}{q}}\\
&\approx \sum_{n\leq N}2^n\left(\int_{\Delta_n}w(y)\left(\int_{a_{n-1}}^y k(y,z)f(z)dz\right)^qdy\right)^{\frac{r}{q}}\\
&+\sum_{n\leq N}2^n\left(\int_{\Delta_n}w(y)\left(\int_0^{a_{n-1}} k(y,z)f(z)dz\right)^qdy\right)^{\frac{r}{q}}=:J_1^r+J_2^r.
\end{align*}
Since $k(y,z)\approx k(y,x)+k(x,z)$ for $y\in \Delta_n$, $x\in \Delta_{n-1}$, $z\in (0,a_{n-1}),$ then $J_2^r$ is estimated as follows
\begin{align*}
&J_2^r\approx\sum_{n\leq N}\int_{a_{n-1}}^{a_n}u(x)dx\left(\int_{\Delta_n}w(y)\left(\int_0^{a_{n-1}} k(y,z)f(z)dz\right)^qdy\right)^{\frac{r}{q}}\\
&\approx\sum_{n\leq N}\int_{a_{n-1}}^{a_n}u(x)\left(\int_{a_n}^{a_{n+1}}w(y)[k(y,x)]^qdy\right)^{\frac{r}{q}}dx\left(\int_0^{a_{n-1}}f\right)^r\\
&+\sum_{n\leq N}\int_{a_{n-1}}^{a_n}u(x)\left(\int_0^{a_{n-1}} k(x,z)f(z)dz\right)^rdx\left(\int_{\Delta_n}w\right)^{\frac{r}{q}}\\
&\lesssim
\sum_{n\leq N}\int_{a_{n-1}}^{a_n}u(x)\left(\int_x^\infty w(y)[k(y,x)]^qdy\right)^{\frac{r}{q}}\left(\int_0^xf\right)^rdx\\
&+\sum_{n\leq N}\int_{a_{n-1}}^{a_n}u(x)\left(\int_x^\infty w\right)^{\frac{r}{q}}\left(\int_0^xk(x,z)f(z)dz\right)^rdx\\
&\lesssim
(A_1^r+A_0^r)\|f\|^r_{L^p_v}.
\end{align*}
For an upper bound of $J_1^r$ we write
\begin{align*}
J_1^r&\approx\sum_{n\leq N}2^n\|H_{a_n,a_{n+1}}f\|_{L^q_w}^r\\
&
\lesssim \sum_{n\leq N}\left(\int_{a_{n-1}}^{a_n}u\right)\|H_{a_n,a_{n+1}}\|_{L^p_v\to L^q_w}^r\left(\int_{a_{n-1}}^{a_{n+1}} f^pv\right)^\frac{r}{p}.
\end{align*}
If $p\le r$ we apply Jensen's inequality and get
$$
J_1\lesssim \sup_{n\leq N}\left(\int_{a_{n-1}}^{a_n}u\right)^\frac{1}{r}\|H_{a_n,a_{n+1}}\|_{L^p_v\to L^q_w}\|f\|_{L^p_v}\le A_2\|f\|_{L^p_v}.
$$
If $r<p$ we apply H\"{o}lder's inequality with the exponents $\frac{s}{r}$ and $\frac{p}{r}$ and obtain
\begin{align*} J_1^s&\lesssim \sum_{n\leq N}\left(\int_{a_{n}}^{a_{n+1}}u\right)^\frac{s}{r}\|H_{a_n,a_{n+1}}\|_{L^p_v\to L^q_w}^s\|f\|_{L^p_v}^s\\
 &\lesssim \sum_{n\leq N}\left(\int_{a_n}^{a_{n+1}}u\right) \left(\int_0^{a_n}u\right)^\frac{s}{p}\|H_{\sigma^{-1}(a_{n+1}),\sigma(a_n)}\|_{L^p_v\to L^q_w}^s\|f\|_{L^p_v}^s\\
 &\le  \sum_{n\leq N}\left(\int_{a_n}^{a_{n+1}}u(x)\left(\int_0^xu\right)^\frac{s}{p}\|H_{\sigma^{-1}(x),\sigma(x)}\|_{L^p_v\to L^q_w}^s\,dx\right)\|f\|_{L^p_v}^s\\
&\le A_2^s\|f\|_{L^p_v}^s.
\end{align*}
Thus,
\beqn
 \left\|Tf\right\|_{L^r_u} \lesssim(A_0+A_1+A_2)\|f\|_{L^p_v}
 \nonumber
\eeqn
and the upper bound $C_T\lesssim A_0+A_1+A_2$ is proved.

{\it Necessity.} Since 
\begin{align*}
(Tf)(x)&\geq\left(\int_x^\infty w(y)\left(\int_0^x k(y,z)f(z)dz\right)^qdy\right)^{\frac{1}{q}}\\
&\gtrsim
\left(\int_x^\infty w\right)^{\frac{1}{q}}\int_0^x k(x,z)f(z)dz,
\end{align*} 
the inequality \eqref{Tin} implies \eqref{A0} and $C_T\gtrsim A_0$.
Moreover,  
\begin{align*}
(Tf)(x)&\ge \left(\int_x^\infty w(y)\left(\int_0^x k(y,z)f(z)dz\right)^qdy\right)^{\frac{1}{q}}\\
&\gtrsim
\left(\int_x^\infty [k(y,x)]^qw(y)dy\right)^{\frac{1}{q}}\int_0^x f.
\end{align*}
Then \eqref{Tin} implies \eqref{A1} and $C_T\gtrsim A_1$.
It follows from \eqref{Tin} that
$$
C_T\|f\|_{L^p_v}\ge \left(\int_0^t u\right)^\frac{1}{r}\|H_tf\|_{L^q_w},\,\,f\in\Msp
$$
for any $t\in(0,\infty).$ Hence,
$$
C_T\ge  \sup_{t\in(0,\infty)}\left(\int_0^t u\right)^\frac{1}{r}\|H_t\|_{L^p_v\to L^q_w}
$$
and the lower bound $C_T\gtrsim A_2$ is proved for $p\le r.$   
Now, let $r<p$. We have
\begin{align*}
A_2^s&=\int_0^\infty u(x)\left(\int_0^x u\right)^\frac{s}{p}\|H_{\sigma^{-1}(x),\sigma(x)}\|^s_{L^p_v\to L^q_w}\,dx\\
&=
\sum_{n\leq N}\int_{a_n}^{a_{n+1}}u(x)\left(\int_0^x u\right)^\frac{s}{p}\|H_{\sigma^{-1}(x),\sigma(x)}\|^s_{L^p_v\to L^q_w}\,dx\\
&\le\sum_{n\leq N}\left(\int_{a_n}^{a_{n+1}}u\right)\left(\int_0^{a_{n+1}}u\right)^\frac{s}{p}\|H_{\sigma^{-1}(a_n),\sigma(a_{n+1})}\|_{L^p_v\to L^q_w}^s\\
&\approx \sum_{n\leq N} (2^n)^\frac{s}{r}\|H_{a_{n-1},a_{n+2}}\|_{L^p_v\to L^q_w}^s=:\bar{A}_2^s.
\end{align*}
Let $\theta\in(0,1)$ be arbitrary. For all $n\leq N$ there is $f_n\in\Msp$ such that $\textrm{supp} f_n\subset [a_{n-2},a_{n+2}]$, $\|f_n\|_{L^p_v}=1$ and $$
\|H_{a_{n-1},a_{n+2}}f_n\|_{L^q_w}\ge \theta\|H_{a_{n-1},a_{n+2}}\|_{L^p_v\to L^q_w}.
$$
Put
$$
g_n:=(2^n)^\frac{s}{pr}\|H_{a_{n-1},a_{n+2}}\|_{L^p_v\to L^q_w}^\frac{s}{p}f_n,~~~g:=\sum_{n\leq N} g_n.
$$
We find
\begin{align*}
\|g\|_{L^p_v}^p&=\sum_{j\leq N}\int_{a_j}^{a_{j+1}}\left(\sum_{n\leq N} g_n(x)\right)^pv(x)\,dx\\
&=\sum_{j\leq N}\int_{a_j}^{a_{j+1}}\left(\sum_{n=j-1}^{j+2} g_n(x)\right)^pv(x)\,dx\\
&\lesssim \sum_{j\leq N}\int_{a_{j-2}}^{a_{j+2}}g_j(x)^pv(x)\,dx\\
&=\sum_{j\leq N}(2^j)^\frac{s}{r}\|H_{a_{j-1},a_{j+2}}\|_{L^p_v\to L^q_w}^s=\bar{A}_2^s.
\end{align*}
Finally, applying \eqref{Tin}
\begin{align*}
C_T^r\bar{A}_2^{\frac{sr}{p}}&\gtrsim C_T^r\|g\|_{L^p_v}^r\geq\int_0^\infty[Tg]^ru
\ge\sum_{n\leq N} \int_{a_{n-2}}^{a_{n-1}} [Tg]^ru\\
&\ge \sum_{n\leq N} \left(\int_{a_{n-2}}^{a_{n-1}} u\right)\|H_{a_{n-1},a_{n+2}}g\|_{L^q_w}^r
\gtrsim\sum_{n\leq N} 2^n\|H_{a_{n-1},a_{n+2}}g_n\|_{L^q_w}^r\\
&=\sum_{n\leq N} (2^n)^\frac{s}{r}\|H_{a_{n-1},a_{n+2}}\|_{L^p_v\to L^q_w}^\frac{sr}{p}\|H_{a_{n-1},a_{n+2}}f_n\|_{L^q_w}^r\\
&\ge \theta^r \sum_{n\leq N} (2^n)^\frac{s}{r}\|H_{a_{n-1},a_{n+2}}\|_{L^p_v\to L^q_w}^s\gtrsim\theta^r\bar{A}_2^s.
\end{align*}
Thus, $C_T\gtrsim \theta\bar{A}_2.$ Hence, $C_T\gtrsim \theta A_2$ and the required lower bound $C_T\gtrsim A_0+A_1+ A_2$ follows.

The case $q=\infty$ is treated similarly with only replacement of \eqref{sum} by a trivial modification
\begin{equation}\label{summ}
\sum_{n\in\mathbb{Z}} 2^n\left(\sup_{i\geq n}\lambda_i\right)^s\approx \sum_{n\in\mathbb{Z}} 2^n\lambda_n^s.
\end{equation}
\end{proof}

\begin{rem} For $p=\infty$ we have 
\beqn\label{Tpinf}
C_T=\left\|T\left(\frac{1}{v}\right)\right\|_{L^r_u}
\eeqn
and for $r=\infty$
\beqn\label{Trinf}
C_T=\sup_{t\geq 0}U(t)\left\|H_t\right\|_{L^p_v\to L^q_w}, 
\eeqn
where $U(t):=\underset{0\leq x\leq t}{\mathrm{ess\,sup}}\,u(x).$
\end{rem}

Now, for $0<c<d\le\infty$ and $f\in\Msp$ we put
\begin{align*}
(H^*_{c,d} f)(x)&:=\chi_{[c,d)}(x)\int_x^{\sigma(d)} k(z,x)f(z)dz,\\
 (H^*_c f)(x)&:=\chi_{[c,\infty)}(x)\int_x^\infty k(z,x)f(z)dz.
\end{align*}

\begin{thm}
Let $1\le p<\infty$,  $0<r<\infty$, $0<q\le\infty$, $\frac{1}{s}:=\left(\frac{1}{r}-\frac{1}{p}\right)_+$. Then the inequality \eqref{Sin} is fulfilled if and only if the inequalities 
\begin{align}
&\left(\int_0^\infty u(x)\left(\int_x^{\sigma^2(x)}w\right)^\frac{r}{q}\left(\int_{\sigma^2(x)}^\infty k(z,\sigma^2(x))f(z)dz\right)^r\,dx\right)^\frac{1}{r}\le \mathbb{A}_0\|f\|_{L^p_v},\label{mA0}\\
&\left(\int_0^\infty u(x)\left(\int_x^{\sigma^2(x)}[k(\sigma^2(x),z)]^qw(z)dz\right)^{\frac{r}{q}}\left(\int_{\sigma^2(x)}^\infty f\right)^r\,dx\right)^\frac{1}{r}\le \mathbb{A}_1\|f\|_{L^p_v},\label{mA1}
\end{align}
if $q<\infty$ or
\begin{align}
&\left(\int_0^\infty u(x)[\underset{y\in(x,\sigma^2(x))}{\mathrm{ess\,sup}}w(y)]^r\left(\int_{\sigma^2(x)}^\infty k(z,\sigma^2(x))f(z)dz\right)^r\,dx\right)^\frac{1}{r}\le \mathbb{A}_0\|f\|_{L^p_v},\label{mA0inf}\\
&\left(\int_0^\infty u(x)[\underset{y\in(x,\sigma^2(x))}{\mathrm{ess\,sup}}[w(y)k(\sigma^2(x),y)]^r\left(\int_{\sigma^2(x)}^\infty f\right)^r\,dx\right)^\frac{1}{r}\le \mathbb{A}_1\|f\|_{L^p_v}\label{mA1inf}
\end{align}
for $q=\infty$ hold for all $f\in \Msp$ and the constant 
\begin{equation}\label{mA2}
 \mathbb{A}_2:=\begin{cases}
    \displaystyle\sup_{t\in(0,\infty)}\left(\int_0^t u\right)^\frac{1}{r}\|H_t^*\|_{L^p_v\to L^q_w}, & p\le r,\\
    \displaystyle \left(\int_0^\infty u(x)\left(\int_0^x u\right)^\frac{s}{p}\|H^*_{\sigma^{-1}(x),\sigma(x)}\|^s_{L^p_v\to L^q_w}\,dx\right)^\frac{1}{s}, & r<p
   \end{cases}
\end{equation}
is finite. Moreover, $C_S\approx \mathbb{A}_0+\mathbb{A}_1+\mathbb{A}_2.$
\end{thm}
\begin{proof}
Let the sequence $\{a_n\}$ be the same as in the proof of Theorem \ref{theorem11} and $q<\infty.$

{\it Sufficiency.} We have
\begin{align*} 
\int_0^\infty [Sf]^ru&=\sum_{n\leq N}\int_{\Delta_n}[Sf]^ru
\approx \sum_{n\leq N}2^n\left(\int_{a_n}^\infty w(y)\left(\int_y^\infty k(z,y)f(z)dz\right)^qdy\right)^{\frac{r}{q}}\\
&\approx \sum_{n\leq N}2^n\left(\int_{\Delta_n}w(y)\left(\int_y^\infty k(z,y)f(z)dz\right)^qdy\right)^{\frac{r}{q}}\\
&\approx \sum_{n\leq N}2^n\left(\int_{\Delta_n}w(y)\left(\int_y^{a_{n+2}} k(z,y)f(z)dz\right)^qdy\right)^{\frac{r}{q}}\\
&+\sum_{n\leq N}2^n\left(\int_{\Delta_n}w(y)\left(\int_{a_{n+2}}^\infty k(z,y)f(z)dz\right)^qdy\right)^{\frac{r}{q}}=:I_1^r+I_2^r.
\end{align*}
Since for $y\in \Delta_n$, $x\in \Delta_{n-1}$, $z\in (a_{n+2},\infty)$ it holds $k(z,y)\approx k(z,\sigma^2(x))+k(\sigma^2(x),y)$, then the term $I_2^r$ is estimated as follows

\begin{align*}
I_2^r&\lesssim
\sum_{n\leq N}\int_{a_{n-1}}^{a_n} u(x)\left(\int_{\Delta_n}w(y)\left(\int_{a_{n+2}}^\infty k(z,y)f(z)dz\right)^qdy\right)^{\frac{r}{q}}dx
\\
&\lesssim\sum_{n\leq N}\int_{a_{n-1}}^{a_n} u(x)\left(\int_{\Delta_n}w\right)^{\frac{r}{q}}\left(\int_{\sigma^2(x)}^\infty k(z,\sigma^2(x))f(z)dz\right)^rdx
\\
&+\sum_{n\leq N}\int_{a_{n-1}}^{a_n} u(x)
\left(\int_{\Delta_n}[k(\sigma^2(x),y)]^qw(y)dy\right)^{\frac{r}{q}}\left(\int_{\sigma^2(x)}^\infty f\right)^rdx\\
&\lesssim\sum_{n\leq N}\int_{a_{n-1}}^{a_n} u(x)\left(\int_x^{\sigma^2(x)}w\right)^{\frac{r}{q}}\left(\int_{\sigma^2(x)}^\infty k(z,\sigma^2(x))f(z)dz\right)^rdx
\\
&+\sum_{n\leq N}\int_{a_{n-1}}^{a_n} u(x)
\left(\int_x^{\sigma^2(x)}[k(\sigma^2(x),y)]^qw(y)dy\right)^{\frac{r}{q}}\left(\int_{\sigma^2(x)}^\infty f\right)^rdx\\
&\leq(\mathbb{A}_0^r+\mathbb{A}_1^r)\left(\int_0^\infty f^pv\right)^\frac{r}{p}.
\end{align*}
To estimate $I_1^r$ we write
$$
I_1^r\lesssim\sum_{n\leq N}\left(\int_{a_{n-1}}^{a_n}u\right)\|H^*_{a_n,a_{n+1}}\|_{L^p_v\to L^q_w}^r\left(\int_{a_n}^{a_{n+2}} f^pv\right)^\frac{r}{p}.
$$
If $p\le r,$ by Jensen's inequality
$$
I_1\lesssim \sup_{n\leq N}
\left(\int_{a_{n-1}}^{a_n}u\right)^\frac{1}{r}\|H^*_{a_n,a_{n+1}}\|_{L^p_v\to L^q_w}\|f\|_{L^p_v}\leq \mathbb{A}_2\|f\|_{L^p_v}.
$$
If $r<p,$ applying H\"{o}lder's inequality with exponents $\frac{s}{r}$ and $\frac{p}{r}$ similar to the proof of Theorem~\ref{theorem11}, we find
\begin{align*}
I_1^s&\lesssim \sum_{n\leq N}
\left(\int_{a_{n}}^{a_{n+1}}u\right)^\frac{s}{r}\|H^*_{a_n,a_{n+1}}\|_{L^p_v\to L^q_w}^s\|f\|_{L^p_v}^s\\
 &\lesssim  \sum_{n\leq N}
\left(\int_{a_n}^{a_{n+1}}u(x)\left(\int_0^xu\right)^\frac{s}{p}\|H^*_{\sigma^{-1}(x),\sigma(x)}\|_{L^p_v\to L^q_w}^s\,dx\right)\|f\|_{L^p_v}^s\\
&\le \mathbb{A}_2^s\|f\|_{L^p_v}^s.
\end{align*}
Thus, $C_S\lesssim \mathbb{A}_0+\mathbb{A}_1+\mathbb{A}_2.$
 
{\it Necessity.} Since 
$$
(Sf)(x)\gtrsim \|\chi_{[x,\sigma^2(x))}\|_{L^q_w}\int_{\sigma^2(x)}^\infty k(z,\sigma^2(x))f(z)dz,
$$ 
then \eqref{Sin} implies \eqref{mA0} and $C_S\gtrsim \mathbb{A}_0$.
Also,  
\begin{align*}
(Sf)(x)&\ge \left(\int_x^{\sigma^2(x)}w(y)\left(\int_{\sigma^2(x)}^\infty k(z,y)f(z)dz\right)^q\right)^{\frac{1}{q}}\\
&\gtrsim
\left(\int_x^{\sigma^2(x)}w(y)k(\sigma^2(x),y)dy\right)^{\frac{1}{q}}\int_{\sigma^2(x)}^\infty f.
\end{align*}
Therefore, \eqref{Sin} implies \eqref{mA1} and $C_S\gtrsim \mathbb{A}_1.$

Now, let $t\in(0,\infty)$ be fixed. It follows from \eqref{Sin}
$$
C_S\|f\|_{L^p_v}\ge \left(\int_0^t u\right)^\frac{1}{r}\|H^*_tf\|_{L^q_w},\,\,f\in\Msp.
$$
Hence,
$$
C_S\ge  \sup_{t\in(0,\infty)}\left(\int_0^t u\right)^\frac{1}{r}\|H_t^*\|_{L^p_v\to L^q_w}
$$
and $C_S\gtrsim \mathbb{A}_2$ for $p\le r$ is shown.   
   
Now, let $r<p$. As in the proof of Theorem~\ref{theorem11} we find
\begin{align*}
\mathbb{A}_2^s&=\int_0^\infty u(x)\left(\int_0^x u\right)^\frac{s}{p}\|H^*_{\sigma^{-1}(x),\sigma(x)}\|^s_{L^p_v\to L^q_w}\,dx\\
&\lesssim \sum_{n\leq N} (2^n)^\frac{s}{r}\|H^*_{a_{n-1},a_{n+2}}\|_{L^p_v\to L^q_w}^s.
\end{align*}
Let $\theta\in(0,1)$ be arbitrary. For all $n\leq N$ there is $f_n\in\Msp$ such that $\textrm{supp} f_n\subset [a_{n-1},a_{n+3}]$, $\|f_n\|_{L^p_v}=1$ and $$
\|H^*_{a_{n-1},a_{n+2}}f_n\|_{L^q_w}\ge \theta\|H^*_{a_{n-1},a_{n+2}}\|_{L^p_v\to L^q_w}.
$$
Put 
$$
g_n:=(2^n)^\frac{s}{pr}\|H^*_{a_{n-1},a_{n+2}}\|_{L^p_v\to L^q_w}^\frac{s}{p}f_n,~~~g:=\sum_{n\leq N} g_n.
$$
Then
\begin{align*}
\|g\|_{L^p_v}^p&=\sum_{j\leq N}\int_{a_j}^{a_{j+1}}\left(\sum_{n\leq N} g_n(x)\right)^pv(x)\,dx\\
&=\sum_{j\leq N}\int_{a_j}^{a_{j+1}}\left(\sum_{n=j-2}^{j+1} g_n(x)\right)^pv(x)\,dx\\
&\lesssim\sum_{j\leq N}\int_{a_{j-1}}^{a_{j+3}}g_j(x)^pv(x)\,dx=
\sum_{j\leq N}(2^j)^\frac{s}{r}\|H^*_{a_{j-1},a_{j+2}}\|_{L^p_v\to L^q_w}^s.
\end{align*}
Now,
$$
\int_0^\infty [Sg]^ru\gtrsim\theta^r \sum_{n\leq N} (2^n)^\frac{s}{r}\|H^*_{a_{n-1},a_{n+2}}\|_{L^p_v\to L^q_w}^s
$$
and we obtain $C_S\gtrsim \mathbb{A}_2$ for $r<p$ and $C_S\gtrsim \mathbb{A}_0+\mathbb{A}_1+\mathbb{A}_2$ similar to the proof of Theorem~\ref{theorem11}.  The case $q=\infty$ is proved analogously. 
\end{proof}

\begin{rem} Precise characterization of the inequalities \eqref{A0}-\eqref{A1inf}, \eqref{mA0}-\eqref{mA1inf}, sharp  estimates of the norms $\left\|H_t\right\|_{L^p_v\to L^q_w},$ $\left\|H_{\sigma^{-1}(x),\sigma(x)}\right\|_{L^p_v\to L^q_w},$  $\left\|H^\ast_t\right\|_{L^p_v\to L^q_w}$ and $\left\|H^\ast_{\sigma^{-1}(x),\sigma(x)}\right\|_{L^p_v\to L^q_w}$ 
are known and can be found (in various, but equivalent forms) by using, for instance, the results of \cite{SU1}, \cite{SU2} and \cite{P4}, where an integral form of criterion for the case $0<q<1$ was found.
\end{rem} 

\begin{rem} For $p=\infty$ we have 
\beqn\label{Spinf}
C_S=\left\|S\left(\frac{1}{v}\right)\right\|_{L^r_u}
\eeqn
and for $r=\infty$
\beqn\label{Srinf}
C_S=\sup_{t\geq 0}U(t)\left\|H^\ast_t\right\|_{L^p_v\to L^q_w}, 
\eeqn
where $U(t):=\underset{0\leq x\leq t}{\mathrm{ess\,sup}}\,u(x).$
\end{rem}

\section{Operators ${\cal T}$ and ${\cal S}$\label{calTS}}

For finding criteria for \eqref{calTin} and \eqref{calSin} we suppose that $0<\int_t^\infty u<\infty$  for all $t>0$ and define the functions $\zeta:[0,\infty)\to [0,\infty)$, $\zeta^{-1}:[0,\infty)\to [0,\infty)$ by
$$
\zeta(x):=\sup\left\{y>0:\int_y^\infty u\ge \frac{1}{2}\int_x^\infty u\right\},$$
$$
\zeta^{-1}(x):=\sup\left\{y>0:\int_y^\infty u\ge 2\int_x^\infty u\right\},
$$
where $\sup\varnothing=0$. Let, also, $\zeta^{m}$, $m\in \Na$ be a composition of $m$ functions $\zeta$ and similar for  $\zeta^{-m}.$ 
For $0\le c<d<\infty$ and $f\in\mathfrak{M}^+$ put
\begin{align*}
({\cal H}_{c,d}f)(x):=\chi_{(c,d]}(x)\int_x^{\zeta(d)}k(z,x)f(z)dz,\\
({\cal H}_d f)(x):=\chi_{(0,d]}(x)\int_x^\infty k(z,x)f(z)dz,\\
({\cal H}^*_{c,d} f)(x):=\chi_{(c,d]}(x)\int_{\zeta^{-1}(c)}^x k(x,z)f(z)dz,\\
 ({\cal H}^*_d f)(x):=\chi_{(0,d]}(x)\int_0^x k(x,z)f(z)dz.
\end{align*}

Similar to the previous section we prove the following theorems.

\begin{thm}\label{theorem31} Let $1\le p<\infty$,  $0<r<\infty$, $0<q\le\infty,$ $\frac{1}{s}:=\left(\frac{1}{r}-\frac{1}{p}\right)_+$.   For validity of the inequality \eqref{calTin} it is necessary and sufficient that the inequalities 
\begin{align}
 &\left(\int_0^\infty u(x)\left(\int_0^x w\right)^{\frac{r}{q}}\left(\int_x^\infty k(z,x)f(z)dz\right)^r\,dx\right)^\frac{1}{r}\le {\cal A}_0\|f\|_{L^p_v},\label{calA0}
\\
&\left(\int_0^\infty u(x)\left(\int_0^x [k(x,y)]^qw(y)dy\right)^{\frac{r}{q}}\left(\int_x^\infty f\right)^r\,dx\right)^\frac{1}{r}\le {\cal A}_1\|f\|_{L^p_v},\nonumber
\end{align}
if $q<\infty$ or
\begin{align*}
&\left(\int_0^\infty u(x)[\underset{y\in(0,x)}{\mathrm{ess\,sup}}~w(y)]^r\left(\int_x^\infty k(z,x)f(z)dz\right)^r\,dx\right)^\frac{1}{r}\le {\cal A}_0\|f\|_{L^p_v},\\
&\left(\int_0^\infty u(x)[\underset{y\in(0,x)}{\mathrm{ess\,sup}}[w(y)k(x,y)]]^r\left(\int_x^\infty f\right)^r\,dx\right)^\frac{1}{r}\le {\cal A}_1\|f\|_{L^p_v}
\end{align*}   
for $q=\infty$ hold for all $f\in \Msp$ and the constant 
\begin{equation}\label{calA2}
{\cal A}_2:=\begin{cases}
    \displaystyle\sup_{t\in(0,\infty)}\left(\int_t^\infty  u\right)^\frac{1}{r}\|{\cal H}_t\|_{L^p_v\to L^q_w}, & p\le r,\\
    \displaystyle \left(\int_0^\infty u(x)\left(\int_x^\infty u\right)^\frac{s}{p}\|{\cal H}_{\zeta^{-1}(x),\zeta(x)}\|^s_{L^p_v\to L^q_w}\,dx\right)^\frac{1}{s}, & r<p,
\end{cases}
\end{equation}
is finite. Moreover, $C_{\cal T}\approx{\cal A}_0+{\cal A}_1+{\cal A}_2.$
\end{thm}

\begin{thm}\label{theorem32} Let $1\le p<\infty$,  $0<r<\infty$, $0<q\le\infty,$ $\frac{1}{s}:=\left(\frac{1}{r}-\frac{1}{p}\right)_+$. For validity of the inequality \eqref{calSin} it is necessary and sufficient that the inequalities 
\begin{align}
 &\left(\int_0^\infty u(x)\left(\int_{\zeta^{-2}(x)}^x w\right)^{\frac{r}{q}}\left(\int_0^{\zeta^{-2}(x)} k(\zeta^{-2}(x),z) f(z)dz\right)^r\,dx\right)^\frac{1}{r}\le {\bf A}_0\|f\|_{L^p_v},\label{bfA0}\\
&\left(\int_0^\infty u(x)\left(\int_{\zeta^{-2}(x)}^x w(y)[k(y,\zeta^{-2}(x))]^qdy\right)^{\frac{r}{q}}\left(\int_0^{\zeta^{-2}(x)} f\right)^r\,dx\right)^\frac{1}{r}\le{\bf A}_1\|f\|_{L^p_v},\nonumber
\end{align}   
if $q<\infty$ or
\begin{align*}
&\left(\int_0^\infty u(x)[\underset{y\in(\zeta^{-2}(x),x)}{\mathrm{ess\,sup}}~w(y)]^r\left(\int_0^{\zeta^{-2}(x)} k(\zeta^{-2}(x),z) f(z)dz\right)^r\,dx\right)^\frac{1}{r}\le {\bf A}_0\|f\|_{L^p_v},\\
&\left(\int_0^\infty u(x)[\underset{y\in(\zeta^{-2}(x),x)}{\mathrm{ess\,sup}}[w(y)k(y,\zeta^{-2}(x))]]^r\left(\int_0^{\zeta^{-2}(x)} f\right)^r\,dx\right)^\frac{1}{r}\le {\bf A}_1\|f\|_{L^p_v}
\end{align*}   
for $q=\infty$ hold for all $f\in \Msp$ and the constant
 \begin{equation}\label{bfA2}
 {\bf A}_2:=\begin{cases}
    \displaystyle\sup_{t\in(0,\infty)}\left(\int_t^\infty u\right)^\frac{1}{r}\|{\cal H}_t^*\|_{L^p_v\to L^q_w}, & p\le r,\\
    \displaystyle \left(\int_0^\infty u(x)\left(\int_x^\infty u\right)^\frac{s}{p}\|{\cal H}^*_{\zeta^{-1}(x),\zeta(x)}\|^s_{L^p_v\to L^q_w}\,dx\right)^\frac{1}{s}, & r<p
   \end{cases}
 \end{equation}
is finite. Moreover, $C_{\cal S}\approx{\bf A}_0+{\bf A}_1+{\bf A}_2.$
\end{thm}

\section{Operators ${\bf T}$ and ${\bf S}$\label{bTS}}

Let the functions $\sigma$ and $\sigma^{-1}$ be the same as in the Section 2. 
For $0<c<d\le\infty$ and $f\in\Msp$ we put
\begin{align*}({\bf H}_{c,d} f)(x)&:=\chi_{[c,d)}(x)\int_{\sigma^{-1}(c)}^x f(z)dz,~~~({\bf H}_c f)(x):=\chi_{[c,\infty)}(x)\int_0^x f(z)dz,\\
({\bf H}^*_{c,d} f)(x)&:=\chi_{[c,d)}(x)\int_x^{\sigma(d)}f(z)dz,~~~({\bf H}^*_c f)(x):=\chi_{[c,\infty)}(x)\int_x^\infty f(z)dz.
\end{align*}

\begin{theorem}\label{theorem41} Let $1\le p<\infty$,  $0<r<\infty$, $0<q\leq\infty$, $\frac{1}{s}:=\left(\frac{1}{r}-\frac{1}{p}\right)_+$. For validity of the inequality \eqref{bTin} it is necessary and sufficient that 
 \begin{equation}\label{B}
 B:=B_0+B_1+B_2<\infty,
\end{equation}
 where $B_0$ and $B_1$ are the least possible constants in the inequalities 
 \begin{equation}\label{B0}
 \left(\int_0^\infty u(x)\left(\int_x^\infty k(y,x)w(y)dy\right)^\frac{r}{q}\left(\int_0^x f\right)^r\,dx\right)^\frac{1}{r}\le B_0\|f\|_{L^p_v},
 \end{equation}
and
 \begin{equation}\label{B1}
  \left(\int_0^\infty u(x)[k(\sigma^2(x),x)]^\frac{r}{q}\left(\int_{\sigma^2(x)}^\infty w(y)\left(\int_0^y f\right)^q dy\right)^\frac{r}{q} dx\right)^\frac{1}{r}\le B_1\|f\|_{L^p_v},
 \end{equation}
if $q<\infty$ or
\begin{align}
&\left(\int_0^\infty u(x)[\underset{y\geq x}{\mathrm{ess\,sup}}~k(y,x)w(y)]^r\left(\int_0^x f\right)^r\,dx\right)^\frac{1}{r}\le B_0\|f\|_{L^p_v},\label{B0inf} \\
&\left(\int_0^\infty u(x)[k(\sigma^2(x),x)]^r\left(\underset{y\geq \sigma^2(x)}{\mathrm{ess\,sup}}~ w(y)\int_0^y f\right)^rdx\right)^\frac{1}{r}\le B_1\|f\|_{L^p_v},
\label{B1inf} 
\end{align}   
for $q=\infty$ and $B_2$ is defined by  
 \begin{equation}\label{B2}
B_2:=\begin{cases}
   \displaystyle \sup_{t>0}\left(\int_0^t u\right)^\frac{1}{r}\|{\bf H}_t\|_{L^p_v\to L^q_{w(\cdot) k(\cdot,t)}},& p\le r,\\
 \displaystyle   \left(\int_0^\infty u(x)\left(\int_0^x u\right)^\frac{s}{p}\|{\bf H}_{\sigma^{-1}(x),\sigma^2(x)}\|^s_{L^p_v\to L^q_{w(\cdot) k(\cdot,\sigma^{-1}(x))}}\,dx\right)^\frac{1}{s},& r<p.
      \end{cases}
\end{equation}
Moreover, $C_{\bf T}\approx B$.
\end{theorem}

\begin{proof}    Let the sequence $\{a_n\}$ be the same as in the proof of Theorem~\ref{theorem11} and $q<\infty.$

{\it Sufficiency.} 
We write
\begin{align*}
J&:=\int_0^\infty [Tf]^ru=\sum_{n\le N}\int_{a_n}^{a_{n+1}}[Tf]^ru
\\
&\approx\sum_{n\le N} 2^n\left(\int_{a_n}^\infty k(y,a_n)w(y)\left(\int_0^y f\right)^q dy\right)^{\frac{r}{q}}\approx J_1+J_2,
\end{align*}
where
\begin{align*}
 J_1&:=\sum_{n\le N} 2^n\left(\int_{a_n}^{a_{n+2}} k(y,a_n)w(y)\left(\int_0^y f\right)^q dy\right)^{\frac{r}{q}},\\
 J_2&:=\sum_{n\le N} 2^n\left(\int_{a_{n+2}}^\infty k(y,a_n)w(y)\left(\int_0^y f\right)^q dy\right)^{\frac{r}{q}}.
\end{align*}

{\it Estimate of $J_1.$} We have
\begin{align*}
J_1&\approx
\sum_{n\le N} 2^n\left(\int_{a_n}^{a_{n+2}} k(y,a_n)w(y)\left(\int_{a_{n-1}}^y f\right)^q dy\right)^{\frac{r}{q}}
\\
&+\sum_{n\le N} 2^n\left(\int_{a_n}^{a_{n+2}} k(y,a_n)w(y)dy\right)^{\frac{r}{q}}\left(\int_0^{a_{n-1}} f\right)^r=J_{1,1}+J_{1,2}.
\end{align*}
For $J_{1,2}$ we write
\begin{align*}
J_{1,2}&\approx
\sum_{n\le N} \int_{a_{n-1}}^{a_n}u(x)dx\left(\int_{a_n}^{a_{n+2}} k(y,a_n)w(y)dy\right)^{\frac{r}{q}}\left(\int_0^{a_{n-1}} f\right)^r\\
&\lesssim 
\int_0^\infty u(x)\left(\int_x^\infty k(y,x)w(y)dy\right)^\frac{r}{q}\left(\int_0^x f\right)^r\,dx\le B_0^r\left(\int_0^\infty f^pv\right)^\frac{r}{p}.
\end{align*}
For $J_{1,1}$ we write
\begin{align*}
J_{1,1}&\approx
\sum_{n\le N} 2^n\left(\int_{a_n}^{a_{n+2}}  k(y,a_n)w(y)\left({\bf H}_{a_n,a_{n+2}}f(y)\right)^q dy\right)^{\frac{r}{q}}\\
&\lesssim 
\sum_{n\le N} 2^n\|{\bf H}_{a_n,a_{n+2}}\|^r_{L^p_v\to L^q_{w(\cdot) k(\cdot,a_n)}}
\left(\int_{a_{n-1}}^{a_{n+2}}f^pv\right)^{\frac{r}{p}}.
\end{align*}
If $p\leq r$, by Jensen's inequality we get 
$$J_{1,1}\lesssim  B_2^r\|f\|^r_{L^p_v}.$$
If $r<p$, by H\"older's inequality,  
\begin{align*}
J_{1,1}&\lesssim \left(\sum_{n\le N} (2^n)^\frac{s}{r}\|{\bf H}_{a_n,a_{n+2}}\|^s_{L^p_v\to L^q_{w(\cdot) k(\cdot,a_n)}}\right)^\frac{r}{s}\|f\|_{L^p_v}^r\\
&\lesssim \left(\sum_{n\le N} \int_{a_n}^{a_{n+1}}u\left(\int_0^{a_n} u\right)^\frac{s}{p}\|{\bf H}_{\sigma^{-1}(a_{n+1}),\sigma^2(a_n)}\|^s_{L^p_v\to L^q_{w(\cdot) k(\cdot,\sigma^{-1}(a_{n+1}))}}\right)^\frac{r}{s}\|f\|_{L^p_v}^r\\
&\lesssim B_2^r\|f\|^r_{L^p_v}.
\end{align*}
Thus,
\begin{equation}\label{ocenkaj1}
 J_1\lesssim (B_0+B_2)^r\|f\|_{L^p_v}^r.
\end{equation}

{\it Estimate of $J_2.$} Denote $h(y):=w(y)\left(\int_0^y f\right)^q$ and  using \eqref{O} we obtain
\begin{align*}
 \int_{a_{n+2}}^\infty &k(y,a_n)h(y)\,dy=\sum_{i\ge n}\int_{a_{i+2}}^{a_{i+3}} k(y,a_n)h(y)\,dy\\
 &\approx\sum_{i\ge n}\int_{a_{i+2}}^{a_{i+3}} k(y,a_{i+1})h(y)\,dy+\sum_{i\ge n}\int_{a_{i+2}}^{a_{i+3}} k(a_{i+1},a_n)h(y)\,dy\\
 &\lesssim\sum_{i\ge n}\int_{a_{i+1}}^{a_{i+3}} k(y,a_{i+1})h(y)\,dy+\sum_{i\ge n}\int_{a_{i+2}}^{a_{i+3}} k(a_{i+1},a_n)h(y)\,dy=:I_{1,n}+I_{2,n}.
\end{align*}
Similar to the proof of \eqref{ocenkaj1} we find
\begin{equation}\label{ocenkaodin}
\sum_{n\le N}2^n I_{1,n}^\frac{r}{q}\approx J_1
\lesssim (B_0+B_2)^r\|f\|_{L^p_v}^r.
 \end{equation}
By \cite{GS2}, Lemma 3.1 there is $\alpha\in(0,1)$ such that
\begin{equation}\label{LGS}
k(a_{i+1},a_n)\lesssim\left(\sum_{j=n}^{i}[k(a_{j+1},a_j)]^\alpha\right)^{\frac{1}{\alpha}},\,\,i\geq n.
\end{equation}
By Minkowskii's inequality 
\begin{align*}
I_{2,n}&\lesssim\sum_{i\ge n}\left(\sum_{j=n}^{i}[k(a_{j+1},a_j)]^\alpha\right)^{\frac{1}{\alpha}}
\int_{a_{i+2}}^{a_{i+3}}h(y)dy\\
&\leq\left(\sum_{j\ge n}[k(a_{j+1},a_j)]^\alpha\left(\int_{a_{j+2}}^\infty h\right)^\alpha\right)^{\frac{1}{\alpha}}.
\end{align*}
Hence,
\begin{align*}
 \sum_{n\le N}2^n I_{2,n}^\frac{r}{q}&\leq \sum_{n\le N}2^n\left(\sum_{j\ge n}[k(a_{j+1},a_j)]^\alpha\left(\int_{a_{j+2}}^\infty h\right)^\alpha\right)^{\frac{r}{q\alpha}}
\\
 &\approx \sum_{n\le N}2^n k(a_{n+1},a_n)^\frac{r}{q}\left(\int_{a_{n+2}}^\infty h\right)^\frac{r}{q}\\
 &\approx \sum_{n\le N} \left[\int_{a_{n-1}}^{a_n} u\right] k(\sigma^2(a_{n-1}),a_n)^\frac{r}{q}\left(\int_{\sigma^2(a_n)}^\infty h\right)^\frac{r}{q}\\
 &\lesssim \int_0^\infty u(x)k(\sigma^2(x),x)^\frac{r}{q}\left(\int_{\sigma^2(x)}^\infty w(y)\left(\int_0^y f\right)^q dy\right)^\frac{r}{q} dx\le B_1^r\|f\|_{L^p_v}^r.
 \end{align*}
In case of $q=\infty$ we write
\begin{align*}
 &\esssup_{y\in [a_{n+2},\infty)} k(y,a_n)h(y)=\sup_{i\ge n}\esssup_{y\in {[a_{i+2},a_{i+3})}} k(y,a_n)h(y)\\
&\lesssim \sup_{i\ge n}\esssup_{y\in {[a_{i+1},a_{i+3})}} k(y,a_{i+1})h(y)
+\sup_{i\ge n} k(a_{i+1},a_n)\esssup_{y\in {[a_{i+2},a_{i+3})}}h(y)
=:I_{1,n}+I_{2,n}.
\end{align*}
The estimate \eqref{ocenkaodin} follows by the same way.
Also we have
\begin{align*}
I_{2,n}&\lesssim\left(\sup_{i\ge n}\sum_{j=n}^{i}[k(a_{j+1},a_j)]^\alpha
\left(\esssup_{y\in [a_{i+2},a_{i+3})}h(y)\right)^\alpha\right)^{\frac{1}{\alpha}}\\
&\leq\left(\sum_{j\ge n}[k(a_{j+1},a_j)]^\alpha\sup_{i\ge n}\left(\esssup_{y\in [a_{i+2},a_{i+3})} h(y)\right)^\alpha\chi_{[n,i]}(j)\right)^{\frac{1}{\alpha}}\\
&=\left(\sum_{j\ge n}[k(a_{j+1},a_j)]^\alpha\left(\esssup_{y\in [a_{j+2},\infty)} h(y)\right)^\alpha\right)^{\frac{1}{\alpha}}
\end{align*}
and the inequality  $\sum_{n\le N}2^n I_{2,n}^\frac{r}{q}\lesssim B_1^r\|f\|_{L^p_v}^r$ follows for this case too.
Thus, 
$$J_2\lesssim (B_0+B_1+B_2)^r\|f\|_{L^p_v}^r$$
and 
the upper bound
$
C_{\bf T}\lesssim B_0+B_1+B_2
$
is proved.

{\it Necessity.} Suppose the inequality \eqref{bTin} hold, that is
\begin{equation}\label{Tin1}
\left(\int_0^\infty\left(\int_x^\infty k(y,x)w(y)\left(\int_0^y f\right)^qdy\right)^{\frac{r}{q}} u\right)^{\frac{1}{r}} \leq C_{\bf T}\left(\int_0^\infty f^pv\right)^{\frac{1}{p}}
\end{equation}
for all $f\in\mathfrak{M}^+.$ Narrowing the integration $(0,y)\to(0,x)$ on the left-hand side, we see, that $C_{\bf T}\geq B_0.$ Analogously, if $(x,\infty)\to (\sigma^2(x),\infty)$,  $k(y,x)\gtrsim  k(\sigma^2(x),x)$, then $C_{\bf T}\geq B_1$. If $(0,\infty)\to(0,t)$, $(x,\infty)\to(t,\infty)$, $k(y,x)\gtrsim  k(y,t)$, then 
\begin{equation}\label{lowboundpler}
C_{\bf T}\geq \left(\int_0^t u\right)^\frac{1}{r}\|{\bf H}_t\|_{L^p_v\to L^q_{w(\cdot) k(\cdot,t)}}
 \end{equation}
 for all $t>0$. Consequently,  $C_{\bf T}\gtrsim B_2$ in case of $p\le r$.

In the case $r<p$ we write 
\begin{align*}
B_2^s&=\int_0^\infty u(x)\left(\int_0^x u\right)^\frac{s}{p}\|{\bf H}_{\sigma^{-1}(x),\sigma^2(x)}\|^s_{L^p_v\to L^q_{w(\cdot) k(\cdot,\sigma^{-1}(x))}}\,dx\\
&=
\sum_{n\leq N}\int_{a_n}^{a_{n+1}}u(x)\left(\int_0^x u\right)^\frac{s}{p}\|{\bf H}_{\sigma^{-1}(x),\sigma^2(x)}\|^s_{L^p_v\to L^q_{w(\cdot) k(\cdot,\sigma^{-1}(x))}}\,dx\\
&\le\sum_{n\leq N}\left(\int_{a_n}^{a_{n+1}}u\right)\left(\int_0^{a_{n+1}}u\right)^\frac{s}{p}\|{\bf H}_{\sigma^{-1}(a_n),\sigma^2(a_{n+1})}\|_{L^p_v\to L^q_{w(\cdot) k(\cdot,\sigma^{-1}(a_n))}}^s\\
&\approx \sum_{n\leq N} (2^n)^\frac{s}{r}\|{\bf H}_{a_{n-1},a_{n+3}}\|_{L^p_v\to L^q_{w(\cdot) k(\cdot,a_{n-1})}}^s=:\bar{B}_2^s.
\end{align*}
Let $\theta\in(0,1)$ be arbitrary. Then for all $n\leq N$ there is $f_n\in\Msp$ such that $\textrm{supp} f_n\subset [a_{n-2},a_{n+3}]$, $\|f_n\|_{L^p_v}=1$ and $$
\|{\bf H}_{a_{n-1},a_{n+3}}f_n\|_{L^q_{w(\cdot) k(\cdot,a_{n-1})}}\ge \theta\|{\bf H}_{a_{n-1},a_{n+3}}\|_{L^p_v\to L^q_{w(\cdot) k(\cdot,a_{n-1})}}.
$$
Put
$$
g_n:=(2^n)^\frac{s}{pr}\|{\bf H}_{a_{n-1},a_{n+3}}\|_{L^p_v\to L^q_{w(\cdot) k(\cdot,a_{n-1})}}^\frac{s}{p}f_n,~~~g:=\sum_{n\leq N} g_n.
$$
We have
\begin{align*}
\|g\|_{L^p_v}^p&=\sum_{j\leq N}\int_{a_j}^{a_{j+1}}\left(\sum_{n\leq N} g_n(x)\right)^pv(x)\,dx\\
&=\sum_{j\leq N}\int_{a_j}^{a_{j+1}}\left(\sum_{n=j-2}^{j+2} g_n(x)\right)^pv(x)\,dx\\
&\lesssim \sum_{j\leq N}\int_{a_{j-2}}^{a_{j+3}}g_j(x)^pv(x)\,dx\\
&=\sum_{j\leq N}(2^j)^\frac{s}{r}\|{\bf H}_{a_{j-1},a_{j+3}}\|_{L^p_v\to L^q_{w(\cdot) k(\cdot,a_{j-1})}}^s=\bar{B}_2^s.
\end{align*}
Finally, applying \eqref{Tin}
\begin{align*}
C_{\bf T}^r\mathcal{A}^{\frac{sr}{p}}&\gtrsim C_{\bf T}^r\|g\|_{L^p_v}^r\geq\int_0^\infty[Tg]^ru
\ge\sum_{n\leq N} \int_{a_{n-2}}^{a_{n-1}} [{\bf T}g]^ru\\
&\ge \sum_{n\leq N} \left(\int_{a_{n-2}}^{a_{n-1}} u\right)\|{\bf H}_{a_{n-1},a_{n+3}}g\|_{L^q_{w(\cdot) k(\cdot,a_{n-1})}}^r
\gtrsim\sum_{n\leq N} 2^n\|{\bf H}_{a_{n-1},a_{n+3}}g_n\|_{L^q_{w(\cdot) k(\cdot,a_{n-1})}}^r\\
&=\sum_{n\leq N} (2^n)^\frac{s}{r}\|{\bf H}_{a_{n-1},a_{n+3}}\|_{L^p_v\to L^q_{w(\cdot) k(\cdot,a_{n-1})}}^\frac{sr}{p}\|{\bf H}_{a_{n-1},a_{n+3}}f_n\|_{L^q_{w(\cdot) k(\cdot,a_{n-1})}}^r\ge \theta^r\bar{B}_2^s.
\end{align*}
Thus, $C_{\bf T}\gtrsim \theta\bar{B}_2$. Hence, $C_{\bf T}\gtrsim \theta B_2$ and the required lower bound $C_{\bf T}\gtrsim B_0+B_1+B_2$ follows.
\end{proof}


\begin{rem} Similar to \eqref{Tpinf} and \eqref{Trinf} we have 
\beqn\label{bTpinf}
C_{\bf T}=\left\|{\bf T}\left(\frac{1}{v}\right)\right\|_{L^r_u},\,\,p=\infty,
\eeqn
\beqn\label{bTrinf}
C_{\bf T}\approx\sup_{t\geq 0}U(t)\left\|{\bf H}_t\right\|_{L^p_v\to L^q_{w(\cdot)k(\cdot,t)}}, r=\infty.
\eeqn
\end{rem}

\begin{theorem}\label{integkrits} Let $1\le p<\infty$,  $0<r<\infty$, $0<q\leq\infty$, $\frac{1}{s}:=\left(\frac{1}{r}-\frac{1}{p}\right)_+$. For validity of the inequality
 \eqref{bSin} it is necessary and sufficient that 
 \begin{equation}\label{mB}
 {\mathbb B}:={\mathbb B}_0+{\mathbb B}_1+{\mathbb B}_2<\infty,
\end{equation}
 where ${\mathbb B}_0$ and ${\mathbb B}_1$ are the least possible constants in the inequalities 
 \begin{equation}\label{mB0}
 \left(\int_0^\infty u(x)\left(\int_x^{\sigma^3(x)} k(y,x)w(y)dy\right)^\frac{r}{q}\left(\int_{\sigma^3(x)}^\infty f\right)^r\,dx\right)^\frac{1}{r}\le {\mathbb B}_0\|f\|_{L^p_v},
 \end{equation}
and 
 \begin{equation}\label{mB1}
  \left(\int_0^\infty u(x)k(\sigma^2(x),x)^\frac{r}{q}\left(\int_{\sigma^2(x)}^\infty w(y)\left(\int_y^\infty f\right)^q dy\right)^\frac{r}{q} dx\right)^\frac{1}{r}\le {\mathbb B}_1\|f\|_{L^p_v},
 \end{equation}
when $q<\infty$ and 
\begin{equation}\label{mB0inf}
 \left(\int_0^\infty u(x)[\underset{x\leq y\leq \sigma^3(x)}{\mathrm{ess\,sup}}~k(y,x)w(y)]^r \left(\int_{\sigma^3(x)}^\infty f\right)^r\,dx\right)^\frac{1}{r}\le {\mathbb B}_0\|f\|_{L^p_v},
 \end{equation}
and 
 \begin{equation}\label{mB1inf}
  \left(\int_0^\infty u(x)[k(\sigma^2(x),x)]^r\left(\underset{y\geq \sigma^2(x)}{\mathrm{ess\,sup}}\,w(y) \int_y^\infty f\right)^rdx\right)^\frac{1}{r}\le {\mathbb B}_1\|f\|_{L^p_v},
 \end{equation}
if $q=\infty.$ The constant ${\mathbb B}_2$ is given by
 \begin{equation}\label{mB2}
 {\mathbb B}_2:=\begin{cases}
   \displaystyle \sup_{t>0}\left(\int_0^t u\right)^\frac{1}{r}\|{\bf H}^*_t\|_{L^p_v\to L^q_{w(\cdot) k(\cdot,t)}},& p\le r,\\
 \displaystyle   \left(\int_0^\infty u(x)\left(\int_0^x u\right)^\frac{s}{p}\|{\bf H}^*_{\sigma^{-1}(x),\sigma^2(x)}\|^s_{L^p_v\to L^q_{w(\cdot) k(\cdot,\sigma^{-1}(x))}}\,dx\right)^\frac{1}{s},& r<p.
      \end{cases}
\end{equation}
 Moreover, $C_{\bf S}\approx {\mathbb B}$.
\end{theorem}

\begin{proof}  Let the sequence $\{a_n\}$ be the same as in the proof of Theorem~\ref{theorem41} and $q<\infty.$

{\it Sufficiency.} 
We write
\begin{align*}
J&:=\int_0^\infty [{\bf S}f]^ru=\sum_{n\le N}\int_{a_n}^{a_{n+1}}[{\bf S}f]^ru
\\
&\approx\sum_{n\le N} 2^n\left(\int_{a_n}^\infty k(y,a_n)w(y)\left(\int_y^\infty f\right)^q dy\right)^{\frac{r}{q}}\approx J_1+J_2,
\end{align*}
where
\begin{align*}
 J_1&:=\sum_{n\le N} 2^n\left(\int_{a_n}^{a_{n+2}} k(y,a_n)w(y)\left(\int_y^\infty f\right)^q dy\right)^{\frac{r}{q}},\\
 J_2&:=\sum_{n\le N} 2^n\left(\int_{a_{n+2}}^\infty k(y,a_n)w(y)\left(\int_y^\infty f\right)^q dy\right)^{\frac{r}{q}}.
\end{align*}

{\it Estimate of $J_1.$} We have
\begin{align*}
J_1&\approx
\sum_{n\le N} 2^n\left(\int_{a_n}^{a_{n+2}} k(y,a_n)w(y)\left(\int_y^{a_{n+3}} f\right)^q dy\right)^{\frac{r}{q}}
\\
&+\sum_{n\le N} 2^n\left(\int_{a_n}^{a_{n+2}} k(y,a_n)w(y)dy\right)^{\frac{r}{q}}\left(\int_{a_{n+3}}^\infty f\right)^r=J_{1,1}+J_{1,2}.
\end{align*}
For $J_{1,2}$ we write
\begin{align*}
J_{1,2}&\approx
\sum_{n\le N} \int_{a_{n-1}}^{a_n}u(x)dx\left(\int_{a_n}^{a_{n+2}} k(y,a_n)w(y)dy\right)^{\frac{r}{q}}\left(\int_{a_{n+3}}^\infty f\right)^r\\
&\lesssim 
\int_0^\infty u(x)\left(\int_x^{\sigma^3(x)} k(y,x)w(y)dy\right)^\frac{r}{q}\left(\int_{\sigma^3(x)}^\infty f\right)^r\,dx\le {\mathbb B}_0^r\left(\int_0^\infty f^pv\right)^\frac{r}{p}.
\end{align*}
For $J_{1,1}$ we write
\begin{align*}
J_{1,1}&\approx
\sum_{n\le N} 2^n\left(\int_{a_n}^{a_{n+2}}  k(y,a_n)w(y)\left({\bf H}^*_{a_n,a_{n+2}}f(y)\right)^q dy\right)^{\frac{r}{q}}\\
&\lesssim 
\sum_{n\le N} 2^n\|{\bf H}^*_{a_n,a_{n+2}}\|^r_{L^p_v\to L^q_{w(\cdot) k(\cdot,a_n)}}
\left(\int_{a_n}^{a_{n+3}}f^pv\right)^{\frac{r}{p}}.
\end{align*}
If $p\leq r$ then, by Jensen's inequality, we get 
$$J_{1,1}\lesssim {\mathbb B}_2^r\|f\|^r_{L^p_v}.$$
If $r<p$ then, by H\"older's inequality,  
\begin{align*}
J_{1,1}&\lesssim \left(\sum_{n\le N} (2^n)^\frac{s}{r}\|{\bf H}^*_{a_n,a_{n+2}}\|^s_{L^p_v\to L^q_{w(\cdot) k(\cdot,a_n)}}\right)^\frac{r}{s}\|f\|_{L^p_v}^r\\
&\lesssim \left(\sum_{n\le N} \int_{a_n}^{a_{n+1}}u\left(\int_0^{a_n} u\right)^\frac{s}{p}\|{\bf H}^*_{\sigma^{-1}(a_{n+1}),\sigma^2(a_n)}\|^s_{L^p_v\to L^q_{w(\cdot) k(\cdot,\sigma^{-1}(a_{n+1}))}}\right)^\frac{r}{s}\|f\|_{L^p_v}^r\\
&\lesssim {\mathbb B}_2^r\|f\|^r_{L^p_v}.
\end{align*}
Thus 
\begin{equation}\label{ocenkaj1ops}
 J_1\lesssim ({\mathbb B}_0+{\mathbb B}_2)^r\|f\|_{L^p_v}^r.
\end{equation}

{\it Estimate of $J_2.$} Denote $h(y):=w(y)\left(\int_y^\infty f\right)^q$ and  arguing similar to the proof of Theorem~\ref{theorem41} we obtain
$$J_2\lesssim ({\mathbb B}_0+{\mathbb B}_1+{\mathbb B}_2)^r\|f\|_{L^p_v}^r.$$

{\it Necessity.} Suppose that the inequality \eqref{Sin} holds, that is
\begin{equation}\label{Sin1}
\left(\int_0^\infty\left(\int_x^\infty k(y,x)w(y)\left(\int_y^\infty  f\right)^qdy\right)^{\frac{r}{q}} u\right)^{\frac{1}{r}} \leq C_S\left(\int_0^\infty f^pv\right)^{\frac{1}{p}}
\end{equation}
for all $f\in\mathfrak{M}^+.$ Narrowing the integration $(x,\infty)\to(x,\sigma^3(x))$ and $(y,\infty)\to (\sigma^3(x),\infty)$ on the left-hand side, we see, that $C_{\bf S}\geq {\mathbb B}_0.$ Analogously, if $(x,\infty)\to (\sigma^2(x),\infty)$,  $k(y,x)\gtrsim  k(\sigma^2(x),x)$, then $C_{\bf S}\geq {\mathbb B}_1$. 
The proof of $C_{\bf S}\gtrsim {\mathbb B}_2$  is similar to the proof of $C_{\bf T}\gtrsim B_2$.
\end{proof}

\begin{rem} Similar to \eqref{Spinf} and \eqref{Srinf} the equalities 
\beqn\label{bSpinf}
C_{\bf S}=\left\|{\bf S}\left(\frac{1}{v}\right)\right\|_{L^r_u},\,\,p=\infty
\eeqn
and 
\beqn\label{bSrinf}
C_{\bf S}\approx\sup_{t\geq 0}U(t)\left\|{\bf H}^\ast_t\right\|_{L^p_v\to L^q_{w(\cdot)k(\cdot,t)}},\,\, r=\infty
\eeqn
hold true.
\end{rem} 

\section{Operators ${\mathfrak T}$ and ${\mathfrak S}$\label{mfTS}}

Let the functions $\zeta,\zeta^{-1}:[0,\infty)\to [0,\infty)$  be
the same as in the Section \ref{calTS}. 
For $0\le c<d<\infty$ and $f\in\mathfrak{M}^+$ we define operators
\begin{align*}
({\mathfrak H}_{c,d}f)(x):=\chi_{(c,d]}(x)\int_x^{\zeta(d)}f(z)dz,\\
({\mathfrak H}_d f)(x):=\chi_{(0,d]}(x)\int_x^\infty f(z)dz,\\
({\mathfrak H}^*_{c,d} f)(x):=\chi_{(c,d]}(x)\int_{\zeta^{-1}(c)}^x f(z)dz,\\
 ({\mathfrak H}^*_d f)(x):=\chi_{(0,d]}(x)\int_0^x f(z)dz.
\end{align*}
The following theorems are true.

\begin{thm} Let $1\le p<\infty$,  $0<r<\infty$, $0<q\le\infty,$ $\frac{1}{s}:=\left(\frac{1}{r}-\frac{1}{p}\right)_+$.   For validity of the inequality \eqref{fTin} it is necessary and sufficient that the inequalities 
\begin{align*}
 &\left(\int_0^\infty u(x)\left(\int_0^x k(x,y)w(y)dy\right)^{\frac{r}{q}}\left(\int_x^\infty f\right)^r\,dx\right)^\frac{1}{r}\le {\cal B}_0\|f\|_{L^p_v},
\\
&\left(\int_0^\infty u(x)[k(x,\zeta^{-2}(x))]^{\frac{r}{q}}\left(\int_0^{\zeta^{-2}(x)} w(y)\left(\int_y^\infty f\right)^qdy\right)^{\frac{r}{q}}\,dx\right)^\frac{1}{r}\le {\cal B}_1\|f\|_{L^p_v},
\end{align*}
if $q<\infty$ or
\begin{align*}
&\left(\int_0^\infty u(x)[\underset{y\in(0,x)}{\mathrm{ess\,sup}}~k(x,y)w(y)]^r\left(\int_x^\infty f\right)^r\,dx\right)^\frac{1}{r}\le {\cal B}_0\|f\|_{L^p_v},\\
&\left(\int_0^\infty u(x)[k(x,\zeta^{-2}(x))]^r\left(\underset{y\in(0,\zeta^{-2}(x))}{\mathrm{ess\,sup}}w(y)\int_y^\infty f\right)^r\,dx\right)^\frac{1}{r}\le {\cal B}_1\|f\|_{L^p_v}
\end{align*}   
for $q=\infty$ hold for all $f\in \Msp$ and the constant 
$${\cal B}_2:=\begin{cases}
    \displaystyle\sup_{t\in(0,\infty)}\left(\int_t^\infty  u\right)^\frac{1}{r}\|{\mathfrak H}_t\|_{L^p_v\to L^q_{w(\cdot)k(t,\cdot)}}, & p\le r,\\
    \displaystyle \left(\int_0^\infty u(x)\left(\int_x^\infty u\right)^\frac{s}{p}\|{\mathfrak H}_{\zeta^{-1}(x),\zeta^2(x)}\|^s_{L^p_v\to L^q_{w(\cdot)k(\zeta^2(x),\cdot)}}\,dx\right)^\frac{1}{s}, & r<p,
\end{cases}
$$
is finite. Moreover, $C_{\mathfrak T}\approx{\cal B}_0+{\cal B}_1+{\cal B}_2.$
\end{thm}

\begin{thm} Let $1\le p<\infty$,  $0<r<\infty$, $0<q\le\infty,$ $\frac{1}{s}:=\left(\frac{1}{r}-\frac{1}{p}\right)_+$. For validity of the inequality \eqref{fSin} it is necessary and sufficient that the inequalities 
\begin{align*}
 &\left(\int_0^\infty u(x)\left(\int_{\zeta^{-3}(x)}^x k(x,y)w(y)dy\right)^{\frac{r}{q}}\left(\int_0^{\zeta^{-3}(x)} f\right)^r\,dx\right)^\frac{1}{r}\le {\bf B}_0\|f\|_{L^p_v},\\
&\left(\int_0^\infty u(x)[k(x,\zeta^{-2}(x))]^{\frac{r}{q}}\left(\int_0^{\zeta^{-2}(x)} w(y)\left(\int_0^y f\right)^qdy\right)^{\frac{r}{q}}\,dx\right)^\frac{1}{r}\le{\bf B}_1\|f\|_{L^p_v},
\end{align*}   
if $q<\infty$ or
\begin{align*}
&\left(\int_0^\infty u(x)[\underset{y\in(\zeta^{-3}(x),x)}{\mathrm{ess\,sup}}~k(x,y)w(y)]^r\left(\int_0^{\zeta^{-3}(x)} f\right)^r\,dx\right)^\frac{1}{r}\le {\bf B}_0\|f\|_{L^p_v},\\
&\left(\int_0^\infty u(x)[k(x,\zeta^{-2}(x))]^r\left(\underset{y\in(0,\zeta^{-2}(x))}{\mathrm{ess\,sup}}w(y)\int_0^y f\right)^r\,dx\right)^\frac{1}{r}\le {\bf B}_1\|f\|_{L^p_v}
\end{align*}   
for $q=\infty$ hold for all $f\in \Msp$ and the constant
 $${\bf B}_2:=\begin{cases}
    \displaystyle\sup_{t\in(0,\infty)}\left(\int_t^\infty u\right)^\frac{1}{r}\|{\mathfrak H}_t^*\|_{L^p_v\to L^q_{w(\cdot)k(t,\cdot)}}, & p\le r,\\
    \displaystyle \left(\int_0^\infty u(x)\left(\int_x^\infty u\right)^\frac{s}{p}\|{\mathfrak H}^*_{\zeta^{-1}(x),\zeta^2(x)}\|^s_{L^p_v\to L^q_{w(\cdot)k(\zeta^2(x),\cdot)}}\,dx\right)^\frac{1}{s}, & r<p
   \end{cases}
 $$
is finite. Moreover, $C_{\mathfrak S}\approx{\bf B}_0+{\bf B}_1+{\bf B}_2.$
\end{thm}

\section{$\Gamma^p(v)\to\Gamma^q(w)$ boundedness of the  maximal operator\label{M}}

The maximal Hardy-Littlewood operator is defined by
$$
Mf(x):=\underset{B}\sup\frac{1}{{\rm mes}B}\int_B|f(y)|dy
$$
where the supremum is taken over all balls centered at $x\in\mathbb{R}^n.$ The Lorentz $\Gamma-$spaces were introduced by E.T. Sawyer
\cite{Sa} while working on characterization of the
boundedness of the maximal operator in the weighted Lorentz spaces (see also, for instance, related papers \cite{CPSS}, \cite{CS}, \cite{GS1}, \cite{GS2}, \cite{Sin}, \cite{S3}).
More exactly, if $v\in\mathfrak{M}^+$ and $0<p<\infty,$ then
$$
\Gamma^p(v)=\left\{f\,\,\text {measurable\,\, on}\,\,{\mathbb R}^n: \left(\int_0^\infty[f^{\ast\ast}(x)]^pv(x)dx\right)^\frac{1}{p}<\infty\right\},
$$
where $f^{\ast\ast}(x):=\frac{1}{x}\int_0^xf^\ast(t)dt$ and
$$
f^\ast(t):=\inf\{s>0:{\rm mes}\{x:|f(x)|>s\}\leq t\}.
$$ 
It is known (\cite{BS}, Theorem 3.8) that
\begin{equation}\label{0.10}
[Mf]^\ast(x)\approx \frac{1}{x}\int_0^x f^\ast.
\nonumber
\end{equation} 
Therefore, $M:\Gamma^p(v)\to\Gamma^q(u)$ boundedness is equivalent to the weighted inequality
\begin{equation}\label{Min}
\left(\int_0^\infty \left(\frac{1}{x}\int_0^x\left(\frac{1}{y}\int_0^y f\right)dy\right)^qu(x)dx\right)^{\frac{1}{q}}\leq C \left(\int_0^\infty \left(\frac{1}{t}\int_0^t f\right)^pv(t)dt\right)^{\frac{1}{p}},\,\, f\in\mathfrak{M}^\downarrow
\end{equation}
restricted on the cone $\in\mathfrak{M}^\downarrow\subset\mathfrak{M}^+$ of all nonincreasing functions. Moreover, the least possible constant $C$ is equivalent to the norm of $M$ 
$$ 
C\approx\|M\|_{\Gamma^p(v)\to\Gamma^q(u)}:=\sup_{0\not=f\in\Gamma^p(v)}\frac{\|Mf\|_{\Gamma^q(u)}}{\|f\|_{\Gamma^p(v)}}.
$$
The inequality \eqref{Min} was first characterized in the case $1<p=q<\infty, u=v$ (\cite{S4}, Theorem 5.1) and for $1<p,q<\infty, u\not=v$ in (\cite{GHS}, Theorem 3.3) and (\cite{S}, Theorem 5.1) (see, also \cite{GS}).

Applying Theorems \ref{theorem31}  and \ref{theorem32}  we solve the problem for all $0<p,q<\infty$ and similar to \cite{S} our criteria have an explicit integral form.

Let $\Omega_{1,0}:=\{g\in\mathfrak{M}^\downarrow, tg(t)\in\mathfrak{M}^\uparrow\}.$  
Then $F(t)=\frac{1}{t}\int_0^t f\in\Omega_{1,0}$ for any $f\in\mathfrak{M}^\downarrow$ and $F^p\in\Omega_{p,0}:=\{g\in\mathfrak{M}^\downarrow, t^pg(t)\in\mathfrak{M}^\uparrow\}.$ By the change $G=F^p$ \eqref{Min} becomes equivalent to 
\begin{equation}\label{Gin}
\left(\int_0^\infty \left(\frac{1}{x}\int_0^x G^{\frac{1}{p}}\right)^qu(x)dx\right)^{\frac{p}{q}}\leq C^p \int_0^\infty Gv,\,\, G\in\Omega_{p,0}
\end{equation}
and applying (\cite{S}, Lemma 2.3) we reduce \eqref{Gin} to the inequality
\begin{equation}\label{Hin}
\left(\int_0^\infty \left(\frac{1}{x}\int_0^x \left(\int_0^\infty\frac{h(z)dz}{y^p+z^p}\right)^{\frac{1}{p}}dy\right)^qu(x)dx\right)^{\frac{p}{q}}\lesssim C^p \int_0^\infty hV,\,\, h\in\mathfrak{M}^+,
\end{equation}
where
\begin{equation*}
V(z)=\int_0^\infty\frac{v(y)dy}{y^p+z^p}.
\end{equation*}
Since
$$
\int_0^\infty\frac{h(z)dz}{y^p+z^p}\approx\int_y^\infty\frac{h(z)dz}{z^p}+\frac{1}{y^p}\int_0^y h(z)dz, 
$$
\eqref{Hin} is characterized by the following pair of inequalities:
\begin{equation*}
\left(\int_0^\infty \left(\frac{1}{x}\int_0^x \left(\int_y^\infty h(z)dz\right)^{\frac{1}{p}}dy\right)^qu(x)dx\right)^{\frac{p}{q}}\leq C_1^p \int_0^\infty h(t)t^pV(t)dt,\,\, h\in\mathfrak{M}^+
\end{equation*}
and
\begin{equation*}
\label{H2}
\left(\int_0^\infty \left(\frac{1}{x}\int_0^x \left(\int_0^y h(z)dz\right)^{\frac{1}{p}}\frac{dy}{y}\right)^qu(x)dx\right)^{\frac{p}{q}}\leq C_2^p \int_0^\infty hV,\,\, h\in\mathfrak{M}^+,
\end{equation*}
which are of the form \eqref{calTin} and \eqref{calSin}, respectively. Moreover, 
$$
C\approx C_1+C_2.
$$
Hence, applying Theorems 3.1 and 3.2, we see that 
\begin{equation}\label{C1}
C_1\approx {\mathcal A}_0+{\mathcal A}_2
\end{equation} 
and
\begin{equation}\label{C2}
C_2\approx {\bf A}_0+{\bf A}_2,
\end{equation} 
where the constants A's are defined by \eqref{calA0} and \eqref{calA2} for \eqref{C1} and by \eqref{bfA0} and \eqref{bfA2} for \eqref{C2} under related changes of weights, the function $\zeta$ and auxiliary operators. 

Suppose for simplicity that $0<\int_t^\infty s^{-q}u(s)ds<\infty$ for all $t>0.$
Now, the functions $\zeta$ and $\zeta^{-1}$ are defined by 
$$
\zeta(x):=\sup\left\{y>0:\int_y^\infty s^{-q}u(s)ds\ge \frac{1}{2}\int_x^\infty s^{-q}u(s)ds\right\},$$
$$
\zeta^{-1}(x):=\sup\left\{y>0:\int_y^\infty s^{-q}u(s)ds\ge 2\int_x^\infty s^{-q}u(s)ds\right\}.
$$
For $0\le c<d<\infty$ and $h\in\mathfrak{M}^+$ we put
\begin{align*}
({\cal H}_{c,d}h)(x):=\chi_{(c,d]}(x)\int_x^{\zeta(d)}h,\\
({\cal H}_d h)(x):=\chi_{(0,d]}(x)\int_x^\infty h,\\
({\cal H}^*_{c,d} h)(x):=\chi_{(c,d]}(x)\int_{\zeta^{-1}(c)}^x h,\\
 ({\cal H}^*_d h)(x):=\chi_{(0,d]}(x)\int_0^x h.
\end{align*}
By Theorem \ref{theorem31} ${\cal A}_0$ is the least possible constant in the inequaity
\begin{equation*}
\left(\int_0^\infty u(x)\left(\int_x^\infty h\right)^\frac{q}{p}\,dx\right)^\frac{p}{q}\le {\cal A}_0^p\int_0^\infty h(z)z^pV(z)dz,\,\,h\in\mathfrak{M}^+
\end{equation*}
 and  ${\cal A}_2$ is defined by 
\begin{equation}\label{clA2}
{\cal A}_2^p:=\begin{cases}
    \displaystyle\sup_{t\in(0,\infty)}\left(\int_t^\infty  s^{-q}u(s)ds\right)^\frac{p}{q}\|{\cal H}_t\|_{L^1_{z^pV(z)}\to L^{\frac{1}{p}}}, & p\le q,\\
    \displaystyle \left(\int_0^\infty x^{-q}u(x)\left(\int_x^\infty s^{-q}u(s)ds\right)^\frac{q}{p-q}\|{\cal H}_{\zeta^{-1}(x),\zeta(x)}\|^\frac{q}{p-q}_{L^1_{z^pV(z)}\to L^\frac{1}{p}}dx\right)^\frac{p-q}{q}, & q<p.
\end{cases}
\end{equation}
Also, by Theorem \ref{theorem32} ${\bf A}_0$ is the best possible constant in the inequaity
\begin{equation*}
\left(\int_0^\infty x^{-q}u(x)\left(\log\frac{x}{\zeta^{-2}(x)}\right)^q\left(\int_0^{\zeta^{-2}(x)} h\right)^\frac{q}{p}\,dx\right)^\frac{p}{q}\le {\bf A}_0^p\int_0^\infty hV,\,\,h\in\mathfrak{M}^+
\end{equation*}
 and  ${\bf A}_2$ is determined from 
\begin{equation}\label{bA2}
{\bf A}_2^p:=\begin{cases}
    \displaystyle\sup_{t\in(0,\infty)}\left(\int_t^\infty  s^{-q}u(s)ds\right)^\frac{p}{q}\|{\cal H}_t^\ast\|_{L^1_V\to L^\frac{1}{p}_{\frac{1}{y}}}, & p\le q,\\
    \displaystyle \left(\int_0^\infty x^{-q}u(x)\left(\int_x^\infty s^{-q}u(s)ds\right)^\frac{q}{p-q}\|{\cal H}_{\zeta^{-1}(x),\zeta(x)}^\ast\|^\frac{q}{p-q}_{L^1_V\to L^\frac{1}{p}_{\frac{1}{y}}}dx\right)^\frac{p-q}{q}, & q<p.
\end{cases}
\end{equation}
By well known results (\cite{KA}, Chapter XI, \S\, 1.5, Theorem 4, see also \cite{GS2}, Theorem 1.1) and (\cite{SS}, Theorem 3.3) we have
\begin{equation}\label{cllA0}
{\cal A}_0^p=\sup_{t>0}\frac{\left(\int_0^t u\right)^\frac{p}{q}}{t^pV(t)},\,\,p\leq q
\end{equation}
and
\begin{equation}\label{clllA0}
{\cal A}_0^p\approx\left(\int_0^\infty\left[t^pV(t)\right]^\frac{q}{q-p}
\left(\int_0^t u\right)^\frac{q}{p-q}u(t)dt\right)^\frac{p-q}{q},\,\,q<p.
\end{equation}
Analogously, we find
\begin{equation}\label{bffA0}
{\bf A}_0^p=\sup_{t>0}\frac{\left(\int_{\zeta^2(t)}^\infty x^{-q}u(x)\left(\log\frac{x}{\zeta^{-2}(x)}\right)^q\,dx\right)^\frac{p}{q}}{V(t)},\,\,p\leq q
\end{equation}
and for $q<p$
\begin{equation}\label{bfffA0}
{\bf A}_0^p\approx
\left(\int_0^\infty
\left(\frac{\int_x^\infty s^{-q}u(s)\left(\log\frac{s}{\zeta^{-2}(s)}\right)^qds}{V(\zeta^{-2}(x))}\right)^\frac{q}{p-q}x^{-q}u(x)\left(\log\frac{x}{\zeta^{-2}(x)}\right)^qdx
\right)^\frac{p-q}{q}.
\end{equation}
Again, applying (\cite{KA}, Chapter XI, \S\, 1.5, Theorem 4) and (\cite{SS}, Theorem 3.3) we obtain
\begin{equation*}
\|{\cal H}_t\|_{L^1_{z^pV(z)}\to L^{\frac{1}{p}}}=[V(t)]^{-1},\,\,0<p\leq 1
\end{equation*}
and
\begin{equation*}
\|{\cal H}_t\|_{L^1_{z^pV(z)}\to L^{\frac{1}{p}}}\approx\left(\int_0^t[V(x)]^\frac{1}{1-p}\frac{dx}{x}\right)^{p-1},\,\,p>1,
\end{equation*}
so that it follows from \eqref{clA2} for $p\leq q$
\begin{equation}\label{clA21}
{\cal A}_2=\sup_{t\in(0,\infty)}\left(\int_t^\infty  s^{-q}u(s)ds\right)^\frac{1}{q}
[V(t)]^{-\frac{1}{p}},\,\,0<p\leq 1,
\end{equation}
and
\begin{equation}\label{clA22}
{\cal A}_2\approx\sup_{t\in(0,\infty)}\left(\int_t^\infty  s^{-q}u(s)ds\right)^\frac{1}{q}
\left(\int_0^t[V(x)]^\frac{1}{1-p}\frac{dx}{x}\right)^\frac{1}{p^\prime},  p>1,
\end{equation}
where $p^\prime:=\frac{p}{p-1}.$ By the same way,
\begin{equation*}
\|{\cal H}_{\zeta^{-1}(x),\zeta(x)}\|_{L^1_{z^pV(z)}\to L^\frac{1}{p}}=\left[\frac{\zeta(x)-\zeta^{-1}(x)}{\zeta(x)}\right]^p\frac{1}{V(\zeta(x))},\,\,0<p\leq 1
\end{equation*}
and
\begin{equation*}
\|{\cal H}_{\zeta^{-1}(x),\zeta(x)}\|_{L^1_{z^pV(z)}\to L^\frac{1}{p}}\approx
\left(\int_{\zeta^{-1}(x)}^{\zeta(x)}[t^pV(t)]^{\frac{1}{1-p}}(t-\zeta^{-1}(x))^\frac{1}{p-1}dt\right)^{p-1},\,\,p>1.
\end{equation*}
Hence, from \eqref{clA2} we see that for $q<p$
\begin{equation}\label{clA211}
{\cal A}_2\approx
\left(\int_0^\infty x^{-q}u(x)\left(\int_x^\infty s^{-q}u(s)ds\right)^\frac{q}{p-q}\left[\frac{(\zeta(x)-\zeta^{-1}(x)))}{\zeta(x)[V(\zeta(x))]^\frac{1}{p}}\right]^\frac{pq}{p-q}dx\right)^\frac{p-q}{pq},
\end{equation}
if $0<p\leq 1$ and
\begin{align}\label{clA221}
&{\cal A}_2\approx\nonumber\\
&\left(\int_0^\infty x^{-q}u(x)\left(\int_x^\infty s^{-q}u(s)ds\right)^\frac{q}{p-q}\left(\int_{\zeta^{-1}(x)}^{\zeta(x)}\left(\frac{t-\zeta^{-1}(x)}{t^pV(t)}\right)^\frac{1}{p-1}dt\right)^\frac{q(p-1)}{p-q}dx\right)^\frac{p-q}{pq},
\end{align}
when $p>1.$

Similarly,
\begin{equation*}
\|{\cal H}^\ast_t\|_{L^1_V\to L^\frac{1}{p}_{\frac{1}{y}}}=\sup_{s\in(0,t)}[V(s)]^{-1}\left(\log\frac{t}{s}\right)^p,\,\,0<p\leq 1
\end{equation*}
and
\begin{equation*}
\|{\cal H}^\ast_t\|_{L^1_V\to L^\frac{1}{p}_{\frac{1}{y}}}\approx\left(\int_0^t[V(x)]^\frac{1}{1-p}\left(\log\frac{t}{s}\right)^\frac{1}{p-1}\frac{dx}{x}\right)^{p-1},\,\,p>1.
\end{equation*}
Now, it follows from \eqref{bA2} for $p\leq q$
\begin{equation}\label{bA21}
{\bf A}_2=\sup_{t\in(0,\infty)}\left(\int_t^\infty  s^{-q}u(s)ds\right)^\frac{1}{q}\sup_{s\in(0,t)}[V(s)]^{-\frac{1}{p}}\log\frac{t}{s},\,\,0<p\leq 1
\end{equation}
and
\begin{equation}\label{bA22}
{\bf A}_2\approx\sup_{t\in(0,\infty)}\left(\int_t^\infty  s^{-q}u(s)ds\right)^\frac{1}{q}\left(\int_0^t[V(x)]^\frac{1}{1-p}\left(\log\frac{t}{s}\right)^\frac{1}{p-1}\frac{dx}{x}\right)^{\frac{1}{p^\prime}},  p>1,
\end{equation}
We have
\begin{equation*}
\|{\cal H}^\ast_{\zeta^{-1}(x),\zeta(x)}\|_{L^1_V\to L^\frac{1}{p}_{\frac{1}{y}}}=
\sup_{s\in(\zeta^{-1}(x),\zeta(x))}\frac{\left(\log\frac{\zeta(x)}{s}\right)^p}{V(s)}
,\,\,0<p\leq 1
\end{equation*}
and
\begin{equation*}
\|{\cal H}^\ast_{\zeta^{-1}(x),\zeta(x)}\|_{L^1_V\to L^\frac{1}{p}_{\frac{1}{y}}}\approx
\left(\int_{\zeta^{-1}(x)}^{\zeta(x)}[V(t)]^{\frac{1}{1-p}}\left(\log\frac{\zeta(x)}{t}\right)^\frac{1}{p-1}\frac{dt}{t}\right)^{p-1},\,\,p>1.
\end{equation*}
Thus, from \eqref{bA2} we find for $q<p$
\begin{equation}\label{bA211}
{\bf A}_2\approx
\left(\int_0^\infty x^{-q}u(x)\left(\int_x^\infty s^{-q}u(s)ds\right)^\frac{q}{p-q}\left[\sup_{s\in(\zeta^{-1}(x),\zeta(x))}\frac{\left(\log\frac{\zeta(x)}{s}\right)^p}{V(s)}\right]^\frac{q}{p-q}dx\right)^\frac{p-q}{pq},
\end{equation}
if $0<p\leq 1$ and
\begin{align}\label{bA221}
&{\bf A}_2\approx\nonumber\\
&\left(\int_0^\infty x^{-q}u(x)\left(\int_x^\infty s^{-q}u(s)ds\right)^\frac{q}{p-q}\left(\int_{\zeta^{-1}(x)}^{\zeta(x)}\left(\frac{\log\frac{\zeta(x)}{t}}{V(t)}\right)^\frac{1}{p-1}\frac{dt}{t}\right)^\frac{q(p-1)}{p-q}dx\right)^\frac{p-q}{pq},
\end{align}
when $p>1.$

Finally, we obtain the following.
\begin{thm} Let $0<p,q<\infty.$ Then for the maximal Hardy-Littlewood operator 
\begin{equation}
\|M\|_{\Gamma^p(v)\to\Gamma^q(u)}\approx {\cal A}_0+{\cal A}_2+{\bf A}_0+{\bf A}_2,
\end{equation}
where the constants on the right-hand side are determined by \eqref{cllA0}-\eqref{bfffA0} for ${\cal A}_0$ and ${\bf A}_0$ and by \eqref{clA21}-\eqref{bA221} for ${\cal A}_2$ and ${\bf A}_2.$

\end{thm}

\bibliographystyle{abbrv}

\begin{thebibliography}{10}

\bibitem{BS}
C. Bennett and R. Sharpley, "{\it Interpolation of operators}", Pure and Applied Mathematics {\bf 129}, Academic Press, Inc., Boston, MA, 1988.

\bibitem{BK} S. Bloom and R. Kerman, \newblock{ Weighted norm inequalities for
operators of Hardy type}, \newblock {\it Proc. Amer. Math. Soc.} {\bf 113} (1991), 135-141. 

\bibitem{BG} V.I. Burenkov and H.V. Guliev,  
\newblock{ Necessary and sufficient conditions for boundedness of the maximal operator in the local Morry-type spaces,} \newblock {\it Studia Math.} {\bf 163} (2004), 157--176.

\bibitem{BGGM1} V.I. Burenkov, A. Gogatishvili, V.S. Guliev and R.Ch. Mustafayev, 
\newblock{ Necessary and sufficient conditions for boundedness of the fractional maximal operator in the local Morry-type spaces,} \newblock {\it J. Comput. Appl. Math.} {\bf 208} (2007), 280--301.

\bibitem{BGGM2} V.I. Burenkov, A. Gogatishvili, V.S. Guliev and R.Ch. Mustafayev, 
\newblock{Necessary and sufficient conditions for boundedness of the Riesz potential in the local Morry-type spaces,} \newblock {\it Potential Anal.} {\bf 30} (2009), 211--249.

\bibitem{BJT} V.I. Burenkov, P. Jain and T.V. Tararykova, 
\newblock{On boundedness of the Hardy operator in Morry-type spaces,} \newblock {\it Eurasian Math. J.} {\bf 2} (2011), 52--80.

\bibitem{BO} V.I. Burenkov and R. Oinarov,
\newblock{Necessary and sufficient conditions for boundedness of the Hardy-type operator from a weighted Lebesgue space to a Morry-type spaces,}
\newblock {\it Math. Inequal. Appl.} {\bf 16} (2013), 1--19.

\bibitem{CPSS} M. Carro, L. Pick, J. Soria and V.D. Stepanov, On embeddings between classical Lorentz spaces, {\it Math. Inequal. Appl.} {\bf 4} (2001), 397-428.

\bibitem{CS} M. Carro and J. Soria, Boundedness of some integral operators, {\it Canad. J. Math.} {\bf 45} (1993), 1155-1166.

\bibitem{GHS} M.L. Goldman, H.P. Heinig and V.D. Stepanov, \newblock{ On the principle of duality in Lorentz spaces}, \newblock{\it Canad. J. Math.} {\bf 48} (1996), 959--979.

\bibitem{GS}  M.L. Goldman and M.V. Sorokina, Three-weighted Hardy-type inequalities on the cone of quasimonotone functions, {\it Doklady Math.} {\bf 71} (2005), 209--213.

\bibitem{GOP} A. Gogatishvili, B. Opic and L. Pick, \newblock{Weighted inequalities for Hardy-type operators involving suprema}, \newblock {\it Collect. Math.} {\bf 57} (2006), 227--255.

\bibitem{GP} A. Gogatishvili and L. Pick, \newblock{A reduction theorem for supremum operators}, \newblock {\it J. Comp. Appl. Math.} {\bf 208} (2007), 270--279.

\bibitem{GS1} A. Gogatishvili and V.D. Stepanov, \newblock {Reduction theorems for operators on the cones of monotone functions}, {\it J. Math. Anal. Appl.} {\bf 405} (2013), 156--172.

\bibitem{GS2} A. Gogatishvili and V.D. Stepanov, \newblock {Reduction theorems for weighted integral inequalities on the cone of monotone functions}, {\it Russian Math. Surveys}, {\bf 68(4)} (2013), 597--664. 

\bibitem{GMP1} A. Gogatishvili, R. Mustafayev and L.-E. Persson, \newblock{Some new iterated Hardy-type inequalities}, \newblock {\it J. Function Spaces Appl.} 2013, Art. ID 734194, 30 pp.

\bibitem{GMP2} A. Gogatishvili, R. Mustafayev and L.-E. Persson, \newblock{Some new iterated Hardy-type inequalities: the case $\theta=1$}, \newblock {\it J. Inequal. Appl.} 2013, 2013:515.

\bibitem{KA} L.V. Kantorovich and G.P. Akilov, {\it Functional Analysis}, Pergamon Press, Oxford, 1982.

\bibitem{Lai} Q. Lai, \newblock{ Weighted modular inequalities for Hardy-type operators}, \newblock {\it Proc. London Math. Soc.} {\bf 79} (1999), 649--672.

\bibitem{LS} E.N. Lomakina and V.D. Stepanov, \newblock{On the Hardy-type integral operators in Banach function spaces}, \newblock {\it Publ. Mat.} {\bf 42} (1998), 165--194. 

\bibitem{O} R. Oinarov,  \newblock{Two-sided estimates of the norm of some classes of integral operators}, \newblock{\it Proc. Steklov Inst. Math.} {\bf 204} (1994), 205--214.

\bibitem{P0} D.V. Prokhorov, \newblock {Weighted Hardy's inequalities for negative indices}, {\it Publ. Mat.} {\bf 48} (2004), 423--443.

\bibitem{P1} D.V. Prokhorov, \newblock {Inequalities of Hardy type for a class of integral operators with measures}, \newblock{\it Anal. Math.} {\bf 33} (2007), 199--225.

\bibitem{P2} D.V. Prokhorov, \newblock {Inequalities for Riemann-Liouville operator involving suprema}, \newblock{\it Collect. Math.} {\bf 61} (2010), 263--276.

\bibitem{P3} D.V. Prokhorov, \newblock {Lorentz norm inequalities for the Hardy operator involving suprema}, \newblock{\it Proc. Amer. Math. Soc.} {\bf 140} (2012), 1585--1592.

\bibitem{P4} D.V. Prokhorov, \newblock {On a weighted Hardy-type inequality,} \newblock{\it Doklady Math.} {\bf 88} (2013), 687--689.

\bibitem{P5} D.V. Prokhorov, \newblock {Boundedness and compactness of a supremum-involving integral operator}, \newblock{\it Proc. Steklov Inst. Math.} {\bf 283} (2013), 136--148.

\bibitem{PS1} D.V. Prokhorov and V.D. Stepanov, \newblock {Weighted estimates of Riemann-Liouville operators and applications}, \newblock{\it Proc. Steklov Inst. Math.} {\bf 243} (2003), 278--301.

\bibitem{PS2} D.V. Prokhorov and V.D. Stepanov, \newblock {On supremum operators}, \newblock{\it Doklady Math.,} {\bf 84} (2011), 457--458.

\bibitem{PS3} D.V. Prokhorov and V.D. Stepanov, \newblock {Weighted estimates for a class of sublinear operators}, \newblock{\it Doklady Math.,} {\bf 88} (2013), 721--723.

\bibitem{PS4} D.V. Prokhorov and V.D. Stepanov, \newblock { On weighted Hardy inequalities in mixed norms}, \newblock{\it Proc. Steklov Inst. Math.} {\bf 283} (2013), 149--164.

\bibitem{Sa}
E. Sawyer, Boundedness of classical operators on classical Lorentz
spaces, {\it Studia Math.}, {\bf 96} (1990), 145-158.

\bibitem{S}  G. Sinnamon,  Embeddings of concave functions and duals of Lorentz spaces, {\it Publ. Mat.} {\bf 46} (2002), 489--515.

\bibitem{Sin}  G. Sinnamon,  Transferring monotonicity in weighted norm inequalities, {\it Collect. Math.} {\bf 54} (2003), 181--216.

\bibitem{SS}  G. Sinnamon and V.D. Stepanov, The weighted Hardy inequality:
new proofs and the case p=1. {\it J. London Math. Soc.} {\bf 54} (1996), 89--101.

\bibitem{S1} V.D. Stepanov, \newblock{Weighted norm inequalities of Hardy type for a class of integral operators},\newblock {\it J. London Math. Soc.} {\bf 50} (1994), 105--120.

\bibitem{S2} V.D. Stepanov, \newblock{On a supremum operator}, \newblock{\it Spectral Theory, Function Spaces and Inequalities.}  Operator Theory: Advances and Applications. {\bf 219} Birkh\"{a}user. Basel. 2012, 233--242.

\bibitem{S3}  V.D. Stepanov, The weighted Hardy's inequality for
nonincreasing functions. {\it Trans. Amer. Math. Soc.} {\bf 338} (1993), 173--186.

\bibitem{S4} V.D. Stepanov, Integral operators on the cone of monotone
functions. {\it J. London Math. Soc.} {\bf 48} (1993), 465--487.

\bibitem{SU1} V.D. Stepanov and E.P. Ushakova, \newblock{Alternative criteria for the boundedness of Volterra integral operators in Lebesgue spaces}, {\it Math. Inequal. Appl.} {\bf 12} (2009), 873--889.

\bibitem{SU2} V.D. Stepanov and  E.P. Ushakova, \newblock{Kernel operators with variable intervals of integration in Lebesgue spaces and applications}, \newblock{\it Math. Inequal. Appl.} {\bf 13} (2010), 449--510. 

\end{thebibliography}

\vspace{3mm}

Dmitrii V.~Prokhorov, Computing Centre FEB RAS, Khabarovsk, Russia. {\it E-mail address}: prohorov@as.khb.ru
\vspace{3mm}

Vladimir D.~Stepanov, Steklov Institute of Mathematics; Peoples Friendship University of Russia, Moscow, Russia. {\it E-mail address}: stepanov@mi.ras.ru

\end{document}